\documentclass[11pt,a4]{article}

\usepackage{amsthm,amssymb}
\usepackage{amsmath,yhmath}
\usepackage{mathrsfs}
\usepackage{hyperref,graphicx,color}
\usepackage[margin=3cm]{geometry}

\hypersetup{linkcolor=black,citecolor=black,colorlinks=true}

\theoremstyle{plain}
\newtheorem{theorem}{Theorem}[section]
\newtheorem{corollary}[theorem]{Corollary}
\newtheorem{lemma}[theorem]{Lemma}
\newtheorem{proposition}[theorem]{Proposition}

\theoremstyle{definition}

\theoremstyle{remark}

\begin{document}

\title{Hardy Spaces over Half-strip Domains}
\author{Deng Guantie\thanks{
      E-mail: denggt@bnu.edu.cn, School of Mathematical Sciences, 
      Beijing Normal University, Beijing, China.} 
    \ and Liu Rong\thanks{
      Corresponding author, E-mail: rong.liu@mail.bnu.edu.cn
      School of Mathematical Sciences, 
      Beijing Normal University, Beijing, China.}
}

\maketitle

\begin{abstract}
  We define Hardy spaces $H^p(\Omega_\pm)$ on half-strip domain~$\Omega_+$ 
  and $\Omega_-= \mathbb{C}\setminus\overline{\Omega_+}$, where $0<p<\infty$, 
  and prove that functions in $H^p(\Omega_\pm)$ has non-tangential boundary
  limit a.e.\@ on $\Gamma$, the common boundary of $\Omega_\pm$. 
  We then prove that Cauchy integral of functions in $L^p(\Gamma)$ are in 
  $H^p(\Omega_\pm)$, where $1<p<\infty$, that is, Cauchy transform is bounded.
  Besides, if $1\leqslant p<\infty$, then $H^p(\Omega_\pm)$ functions are 
  the Cauchy integral of their non-tangential boundary limits. 
  We also establish an isomorphism between $H^p(\Omega_\pm)$ and 
  $H^p(\mathbb{C}_\pm)$, the classical Hardy spaces over upper and lower half 
  complex planes.
\end{abstract}



\section{Introduction}

Calder\'{o}n studied Cauchy integrals on Lipschitz curves in~\cite{Ca77}, 
and Coifman, Jones and Semmes provided two elementary proofs for boundedness 
on Cauchy transform on Lipschitz curves in~\cite{CJS89}. Kenig gave two 
equivalent definitions for weighted Hardy spaces over Lipschitz domains in 
his doctoral thesis~\cite{Ke80}, and Meyer and Coifman studied some basic 
properties of Hardy spaces over Lipschitz domains in~\cite{MC97}, in order to 
solve one of Calder\'{o}n's problem about generalized Hardy spaces. Let 
$\Gamma$ be a locally rectifiable Jordan curve, $\Omega_\pm$ be the two simply 
connected domains on two sides of $\Gamma$, and we could define two Hardy 
spaces $H^p(\Omega_\pm)$. Calder\'{o}n's problem states that
whether $L^p$ ($1<p<\infty$) functions on $\Gamma$ are sum of two functions 
in $H^p(\Omega_+)$ and $H^p(\Omega_-)$, respectively. However, Meyer and Coifman
only considered upright down boundary limit in their book. More general Hardy
space theories has been researched by Duren in~\cite{Du70} as well.

In our paper~\cite{DL171,DL172}, we adopt Meyer and Coifman's definitions of 
Hardy spaces over Lipschitz domains~$\Omega_\pm$, and proved the exsistence of 
non-tangential boundary limit of $H^p(\Omega_\pm)$ functions. The Cauchy
and ``Poisson'' representations of functions in $H^p(\Omega_\pm)$ 
($1\leqslant p<\infty$) are aslo proved. We offered a characterization of 
$L^p(\Gamma)$ ($1\leqslant p<\infty$) functions to be non-tangential boundaries
of $H^p(\Omega_\pm)$ functions. More importantly, we established an isomorhpism
between $H^p(\Omega_\pm)$ and $H^p(\mathbb{C}_\pm)$, the classical Hardy spaces 
over upper and lower half complex planes.

In this paper, we will change our attention to Hardy spaces over half-strip 
domains, which are still denoted as $\Omega_\pm$ and may be viewed as limit of 
Lipschitz domains, and will prove nearly all
results mentioned above by using similar method, although many adaptations 
must be made. Our definitions of Hardy spaces over half-strip domains are 
influenced by Vinnitskii's paper~\cite{Vi94}, in which proofs of some results 
below are sketched. Besides, as the boundary of half-strip domains are part of
straight lines, the boundedness of Cauchy transform are proved for all 
$1<p<\infty$ by utilizing theorems from $H^p(\mathbb{C_+})$. This is contrast
with the case when $\Omega_\pm$ are Lipschitz domains, where the boundedness 
of Cauchy transform is only proved for $p=2$. Thus, Calder\'{o}n's problem
mentioned above is solved if we consider half-strip domains. However, the 
``Poisson'' represention of functions in $H^p(\Omega_\pm)$ for 
$1\leqslant p<\infty$ are no longer valid in this case.

\section{Basic Definitions}

As usual, the complex plane is denoted as $\mathbb{C}$, and points $w$, $z$ 
on it are denoted as $w=u+\mathrm{i}v$ and $z=x+\mathrm{i}y$, where 
$u$, $v$, $x$, $y$ are in $\mathbb{R}$, the set of real numbers. 
For $s>0$ and $t\in\mathbb{R}$, define half-strip 
$D_{s,t}=\{u+\mathrm{i}v\colon |u|<s, v>t\}$, and its boundary 
\begin{align*}
  \Gamma_{s,t}
  &= \partial D_{s,t}
   = \Gamma_{s,t,1}\cup \Gamma_{s,t,2}\cup \Gamma_{s,t,3}         \\
  &= \{-s+\mathrm{i}v\colon v>t\}\cup \{u+\mathrm{i}t\colon |u|\leqslant s\}
     \cup \{s+\mathrm{i}v\colon v>t\},
\end{align*}
which is oriented in the way that $D_{s,t}$ is on the left side 
of $\Gamma_{s,t}$. Obviously, $D_{s_1,t}\subset D_{s_2,t}$ if $s_1<s_2$, 
and $D_{s,t_1}\subset D_{s,t_2}$ if $t_1>t_2$.

For $0<p\leqslant \infty$ and $F(w)$ defined on $\Gamma_{s,t}$, let 
\[m(s,t,F)
  =\left\{\!\!\begin{array}{ll}
      \bigg(\displaystyle\int_{\Gamma_{s,t}} |F(w)|^p |\mathrm{d}w|
        \bigg)^{\frac1p} & \text{for $0<p<\infty$},                      \\
      \sup\{|F(w)|\colon w\in\Gamma_{s,t}\} & \text{for $p=\infty$},
      \end{array}\right.\]
then Hardy space over the half-strip $D_{s,t}$ is defined as 
\[H^p(D_{s,t})= \{F \text{ is analytic on } D_{s,t}\colon
    \sup_{\substack{0<s_1<s,\\ t_1>t}} m(s_1,t_1,F)< \infty\},\]
and for $F(w)\in H^p(D_{s,t})$, we define the above supremum as 
$\lVert F\rVert_{H^p(D_{s,t})}$ which is called the ``$H^p(D_{s,t})$-norm''
of $F(w)$, while Hardy space over 
$\mathbb{C}\setminus \overline{D_{s,t}}$ is defined as 
\[H^p(\mathbb{C}\setminus \overline{D_{s,t}})
  = \{F \text{ is analytic on } \mathbb{C}\setminus \overline{D_{s,t}}\colon
    \sup_{s<s_1,\, t_1<t} m(s_1,t_1,F)< \infty\},\]
and for $F(w)\in H^p(\mathbb{C}\setminus \overline{D_{s,t}})$, 
its $H^p(\mathbb{C}\setminus \overline{D_{s,t}})$-norm is denoted as 
$\lVert F\rVert_{H^p(\mathbb{C}\setminus \overline{D_{s,t}})}$. Notice that,
the above two $H^p$-norms are really not norm if $0<p<1$, and we choose 
the word ``norm'' only for convenience.

In this paper, we mainly focus on the special cases of $H^p(D_{\sigma,0})$ 
and $H^p(\mathbb{C}\setminus \overline{D_{\sigma,0}})$, 
with $0<p\leqslant\infty$ and $\sigma>0$. We denote $D_{\sigma,0}$ 
as $\Omega_+$, and $\mathbb{C}\setminus \overline{D_{\sigma,0}}$ as $\Omega_-$, 
and their common boundary $\Gamma_{\sigma,0}= \Gamma_{\sigma,0,1}\cup 
  \Gamma_{\sigma,0,2}\cup \Gamma_{\sigma,0,3}$ is denoted as 
  $\Gamma= \Gamma_1\cup \Gamma_2\cup \Gamma_3$. 
It is easy to verify that $H^p(\Omega_\pm)$ are vector spaces equipped with 
norm $\lVert \cdot\rVert_{H^p(\Omega_\pm)}$ if $1\leqslant p\leqslant \infty$,
or with metric $\lVert \cdot\rVert_{H^p(\Omega_\pm)}^p$ if $0<p<1$.

If $1\leqslant p\leqslant \infty$, we denote its conjugate coefficient as $q$, 
that is $\frac1p+\frac1q=1$, then $1\leqslant q\leqslant \infty$. 
If, further, $F(w)\in H^p(\Omega_\pm)$ and $G(w)\in H^q(\Omega_\pm)$, 
we have $F(w)G(w)\in H^1(\Omega_\pm)$ by H\"{o}lder's inequality. 
Let $n$ be a positive integer, then $F(w)\in H^{np}(\Omega_\pm)$ if and only if 
$F^n(w)\in H^p(\Omega_\pm)$.

Our main results of this paper are listed as follows. We first prove in 
Theorem~\ref{thm-170826-1440} that if $1<p<\infty$, the Cauchy transform 
on $\Gamma$ is bounded. Then the existence of non-tangential boundary limit 
of $H^p(\Omega_\pm)$ functions for $1<p<\infty$ is proved in 
Theorem~\ref{thm-170827-1120} and Theorem~\ref{thm-170827-1410}, together 
with the Cauchy integral representation of $H^p(\Omega_\pm)$ functions. 
The existence of non-tangential boundary limit of $H^p(\Omega_\pm)$ functions 
are then extended to the case of $0<p<\infty$ in Theorem~\ref{thm-170827-2300} 
and Theorem~\ref{thm-170827-2301}, and the Cauchy integral representation to the 
case of $1\leqslant p<\infty$. In the end of this paper, 
Theorem~\ref{thm-170914-1100} will give an isomorphism between $H^p(\Omega_\pm)$ 
and $H^p(\mathbb{C}_\pm)$ for $0< p< \infty$. 

\section{Elementary Properties of $H^p(\Omega_\pm)$}

The open disk $\{z\in\mathbb{C}\colon |z-a|<r\}$ where $a\in\mathbb{C}$ 
and $r>0$ is denoted as $D(a,r)$, and the area measure on $\mathbb{C}$ 
is $\mathrm{d}\lambda$. Notice, some results below have already appeared 
in \cite{Vi94}, but usually with little or no proof. We will always provide 
a complete proof when needed.

\begin{lemma}[\cite{Vi94}]\label{lem-170825-1920}
  If $0<p<\infty$, $F(w)\in H^p(\Omega_+)$, and $w=u+\mathrm{i}v\in \Omega_+$, 
  then 
  \[|F(w)|\leqslant \Big(\frac2\pi\Big)^{\frac1p} \lVert F\rVert_{H^p(\Omega_+)} 
      (\min\{\sigma-|u|,v\})^{-\frac1p}.\]
\end{lemma}

\begin{proof}
  Fix $w_0=u_0+\mathrm{i}v_0\in\Omega_+$, and let $\rho=\min\{\sigma-|u|,v\}$, 
  then 
  \[D(w_0,\rho)
    \subset \{u+\mathrm{i}v\colon |u-u_0|<\rho, |v-v_0|<\rho\} 
    \subset \Omega_+.\] 
  Since $|F(w)|^p$ is subharmonic on $\Omega_+$, we have 
  \begin{align*}
    |F(w_0)|^p
    &\leqslant \frac1{\pi\rho^2} \iint_{D(w_0,\rho)} 
        |F(w)|^p\,\mathrm{d}\lambda(w)                            \\
    &\leqslant \frac1{\pi\rho^2} \int_{u_0-\rho}^{u_0+\rho}
        \!\!\int_{v_0-\rho}^{v_0+\rho}
        |F(u+\mathrm{i}v)|^p\,\mathrm{d}v\,\mathrm{d}u              \\
    &\leqslant \frac1{\pi\rho^2}\cdot 2\rho
        \cdot \lVert F\rVert_{H^p(\Omega_+)}^p                  \\
    &= \frac2{\pi\rho} \lVert F\rVert_{H^p(\Omega_+)}^p,
  \end{align*}
  and 
  \[|F(w_0)|\leqslant \Big(\frac2\pi\Big)^{\frac1p} 
      \lVert F\rVert_{H^p(\Omega_+)} \rho^{-\frac1p},\]
  which proves the lemma.
\end{proof}

\begin{lemma}[\cite{Vi94}]\label{lem-170825-2140}
  If $0<p<\infty$, $F(w)\in H^p(\Omega_-)$, and $w=u+\mathrm{i}v\in \Omega_-$, 
  then 
  \[|F(w)|\leqslant\left\{\!\!
       \begin{array}{ll}
          C_p(|u|-\sigma)^{-\frac1p} & \text{if $|u|>\sigma$},\\
          C_p|v|^{-\frac1p} & \text{if $v<0$},
          \end{array}\right.\]
  where $C_p= (2/\pi)^{\frac1p} \lVert F\rVert_{H^p(\Omega_-)}$.
\end{lemma}

The proof is similar to that of Lemma~\ref{lem-170825-1920}, and we should let
$\rho$ be $|u|-\sigma$ if $|u|>\sigma$, and $|v|$ if $v<0$. 
The following theorem shows that $H^p(\Omega_\pm)$ are Banach spaces 
for $1\leqslant p\leqslant\infty$.

\begin{theorem}\label{thm-170827-2120}
  If $0<p<\infty$, then $H^p(\Omega_\pm)$ are complete.
\end{theorem}

\begin{proof}
  Let $\{F_n(w)\}$ be a Cauchy sequence in $H^p(\Omega_+)$, that is 
  \[\lVert F_m- F_n\rVert_{H^p(\Omega_+)}\to 0\quad 
  \text{as } m,n\to 0.\]
  For  $w=u+\mathrm{i}v\in\Omega_+$, by Lemma~\ref{lem-170825-1920},
  \[|F_m(w)- F_n(w)|
    \leqslant \Big(\frac2\pi\Big)^{\frac1p} (\min\{\sigma-|u|,v\})^{-\frac1p}
    \lVert F_m- F_n\rVert_{H^p(\Omega_+)},\]
  then $\{F_n(w)\}$ converges uniformly on compact subset of $\Omega_+$. 
  We denote the convergence function as $F(w)$, 
  which is also analytic on $\Omega_+$.
  
  For any $\varepsilon>0$, there exists $n_0\in\mathbb{N}$, such that if 
  $n>n_0$, then $\lVert F_{n_0}- F_n\rVert_{H^p(\Omega_+)}\leqslant \varepsilon$. 
  By Fatou's lemma, for $0<s<\sigma$ and $t>0$,
  \[\int_{\Gamma_{s,t}} |F(w)-F_{n_0}(w)|^p |\mathrm{d}w|
    \leqslant \lim_{n\to\infty} \int_{\Gamma_{s,t}} 
        |F_n(w)-F_{n_0}(w)|^p |\mathrm{d}w|
    \leqslant \varepsilon^p,\]
  or $\lVert F- F_{n_0}\rVert_{H^p(\Omega_+)}\leqslant \varepsilon$. We then have 
  \[\lVert F\rVert_{H^p(\Omega_+)}
    \leqslant \varepsilon+ \lVert F_{n_0}\rVert_{H^p(\Omega_+)}\quad
    \text{if $1\leqslant p<\infty$,}\]
  or
  \[\lVert F\rVert_{H^p(\Omega_+)}^p
    \leqslant \varepsilon^p+ \lVert F_{n_0}\rVert_{H^p(\Omega_+)}^p\quad
    \text{if $0< p<1$,}\]
  and both of them show that $F(w)\in H^p(\Omega_+)$. Thus, $H^p(\Omega_+)$ 
  is complete. The $H^p(\Omega_-)$ case is similarly proved.  
\end{proof}

Next lemma may be viewed as a refined version of Lemma~\ref{lem-170825-1920}.
\begin{lemma}\label{lem-170825-1930}
  Let $0<p<\infty$, $F(w)\in H^p(\Omega_+)$, $0<s<\sigma$ and $t>0$, 
  then $F(u+\mathrm{i}v)\to 0$ uniformly for $|u|\leqslant s$ 
  as $v\to +\infty$, and 
  \[\lim_{t\to +\infty} \int_{\Gamma_{s,t}} |F(w)|^p|\mathrm{d}w|= 0\]
  for fixed $s$.
\end{lemma}

\begin{proof}
  The first part of the proof is much like that in Lemma~\ref{lem-170825-1920}.
  Let $\rho= \min\{\frac{\sigma-s}2,\frac{t}2\}$, 
  if $w_0=u_0+\mathrm{i}v_0\in \overline{D_{s,t}}$, then
  \[D(w_0,\rho/2)
    \subset \{u+\mathrm{i}v\colon |u-u_0|<\rho, |v-v_0|<\rho\}
    \subset D_{s+\rho,t-\rho} 
    \subset \Omega_+,\]
  and
  \begin{align*}
    |F(u_0+\mathrm{i}v_0)|^p
    &\leqslant \frac4{\pi\rho^2} \iint_{D(w_0,\frac{\rho}2)} 
        |F(w)|^p\,\mathrm{d}\lambda(w)                            \\
    &\leqslant \frac4{\pi\rho^2} \iint_{D_{s+\rho,t-\rho}} 
        \chi_{\{|\mathrm{Im}\,w-v_0|<\rho\}}|F(w)|^p\,\mathrm{d}\lambda(w).
  \end{align*}
  where $\chi_E$ is the characteristic function of a set $E$. 
  By Lebesgue's dominated convergence theorem and
  \[\iint_{D_{s+\rho,t-\rho}} |F(w)|^p\,\mathrm{d}\lambda(w)
    \leqslant (s+\rho) \lVert F\rVert_{H^p(\Omega_+)}^p,\]
  we have $\lim_{v_0\to\infty} |F(u_0+\mathrm{i}v_0)|= 0$, 
  and the uniform convergence is proved.
  
  Suppose $t>1$, then $D_{s,t}\subset D_{s,1}$, and 
  \begin{align*}
    &\lim_{t\to +\infty} \bigg(\int_{\Gamma_{s,t,1}}
        +\int_{\Gamma_{s,t,3}}\bigg) |F(w)|^p |\mathrm{d}w|             \\
    ={}& \lim_{t\to +\infty} \bigg(\int_{\Gamma_{s,1,1}}
            +\int_{\Gamma_{s,1,3}}\bigg) 
            \chi_{\{\mathrm{Im}\,w>t\}} |F(w)|^p |\mathrm{d}w| 
    = 0.
  \end{align*}
  We also have
  \begin{align*}
    \lim_{t\to +\infty} \int_{\Gamma_{s,t,2}} |F(w)|^p |\mathrm{d}w|
    &\leqslant \lim_{t\to +\infty} 2s\max\{|F(w)|\colon w\in\Gamma_{s,t,2}\}\\
    &= \lim_{t\to +\infty} 2s\max\{|F(u+\mathrm{i}t)|\colon |u|\leqslant s\}
     =0,
  \end{align*}
  then
  \[\lim_{t\to +\infty} \int_{\Gamma_{s,t}} |F(w)|^p |\mathrm{d}w|
    = \sum_{j=1}^3 \lim_{t\to +\infty} \int_{\Gamma_{s,t,j}} 
        |F(w)|^p |\mathrm{d}w|
    = 0,\]
  and the lemma is proved.
\end{proof}

\begin{lemma}\label{lem-170825-2130}
  Let $0<p<\infty$, $F(w)\in H^p(\Omega_-)$ and $\sigma<s_1<s_2$, 
  then $F(u+\mathrm{i}v)\to 0$ uniformly for $s_1\leqslant |u|\leqslant s_2$ 
  as $v\to +\infty$.
\end{lemma}

\begin{proof}
  This is a refinement of Lemma~\ref{lem-170825-2140}.
  Let $t_2<t_1<0$, $\rho= \min\{\frac{s_1-\sigma}2,\frac{|t_1|}2\}$, 
  then $D_{s_1,t_1}\subset D_{s_2,t_2}$, and if we choose 
  $w_0=u_0+\mathrm{i}v_0\in \overline{D_{s_2,t_2}}\setminus D_{s_1,t_1}$ 
  with $v_0>|t_1|$, then
  \[D(w_0,\rho/2)
    \subset \{u+\mathrm{i}v\colon |u-u_0|<\rho, |v-v_0|<\rho\}
    \subset D_{s_2+\rho,t_2-\rho}\setminus \overline{D_{s_1-\rho,t_1+\rho}}
    \subset \Omega_-,\]
  and by denoting $D_{s_2+\rho,t_2-\rho}\setminus 
    \overline{D_{s_1-\rho,t_1+\rho}}\cap 
    \{s_1-\rho<|\mathrm{Re}\,w|< s_2+\rho\}$ as $E$, we have
  \begin{align*}
    |F(u_0+\mathrm{i}v_0)|^p
    &\leqslant \frac4{\pi\rho^2} \iint_{D(w_0,\frac{\rho}2)} 
        |F(w)|^p\,\mathrm{d}\lambda(w)                            \\
    &\leqslant \frac4{\pi\rho^2} \iint_E 
        \chi_{\{|\mathrm{Im}\,w-v_0|<\rho\}}|F(w)|^p\,\mathrm{d}\lambda(w).
  \end{align*}
  Now $\lim_{v_0\to\infty} |F(u_0+\mathrm{i}v_0)|= 0$ comes from
  \[\iint_E |F(w)|^p\,\mathrm{d}\lambda(w)
    \leqslant (s_2-s_1+2\rho) \lVert F\rVert_{H^p(\Omega_+)}^p,\]
  and this proves the lemma.
\end{proof}

In order to show that $H^p(\Omega_\pm)$ is not empty for $0<p\leqslant\infty$, 
we need the following lemma.

\begin{lemma}\label{lem-170826-0720}
  If $1< p\leqslant \infty$, $s>0$ and $w_0\notin \Gamma_{s,t}$, define 
  \[F(w)= \frac1{w-w_0}\quad \text{for } w\in\Gamma_{s,t},\]
  then $F(w)\in L^p(\Gamma_{s,t},|\mathrm{d}w|)$.
\end{lemma}

\begin{proof}
  The $p=\infty$ case is obvious. Suppose $1<p<\infty$, and write 
  \[\int_{\Gamma{s,t}} |F(w)|^p |\mathrm{d}w|
    = \int_{\Gamma{s,t}} \frac{|\mathrm{d}w|}{|w-w_0|^p}
    = \sum_{j=1}^3 \int_{\Gamma{s,t,j}} \frac{|\mathrm{d}w|}{|w-w_0|^p}
    = \sum_{j=1}^3 I_j,\]
  where
  \begin{align*}
    I_1&= \int_{\Gamma{s,t,1}} \frac{|\mathrm{d}w|}{|w-w_0|^p}
        = \int_t^{+\infty} \frac{\mathrm{d}v}
               {|-s+\mathrm{i}v-u_0-\mathrm{i}v_0|^p},             \\
    I_2&= \int_{\Gamma{s,t,2}} \frac{|\mathrm{d}w|}{|w-w_0|^p}
        = \int_{-s}^{s} \frac{\mathrm{d}u}
               {|u+\mathrm{i}t-u_0-\mathrm{i}v_0|^p},             \\
    I_3&= \int_{\Gamma{s,t,3}} \frac{|\mathrm{d}w|}{|w-w_0|^p}
        = \int_t^{+\infty} \frac{\mathrm{d}v}
               {|s+\mathrm{i}v-u_0-\mathrm{i}v_0|^p}.
  \end{align*}
  And $w_0\in D_{s,t}$ or $\mathbb{C}\setminus\overline{D_{s,t}}$ 
  since $w_0\notin\Gamma_{s,t}$.
  
  If $w_0=u_0+\mathrm{i}v_0\in D_{s,t}$, then $|u_0|<s$, $v_0>t$, and 
  \begin{align*}
    I_1
    &= \int_t^{+\infty} \frac{\mathrm{d}v}{((s+u_0)^2+(v-v_0)^2)^{\frac{p}2}} 
     \leqslant \int_{\mathbb{R}} 
            \frac{\mathrm{d}v}{(s+u_0)^{p-1}(1+v^2)^{\frac{p}2}}    \\
    &= \frac2{(s+u_0)^{p-1}} \int_{\mathbb{R}^+} 
            \frac{\mathrm{d}v}{(1+v^2)^{\frac{p}2}}
     = \frac1{(s+u_0)^{p-1}} \int_{\mathbb{R}^+} 
            \frac{v^{-\frac12}\mathrm{d}v}{(1+v)^{\frac{p}2}}.
  \end{align*}
  after making proper change of variables. Let $x=\frac{v}{v+1}$ for 
  $v\in\mathbb{R}_+$, then $x\in(0,1)$ and 
  \[\int_{\mathbb{R}^+} \frac{v^{-\frac12}\mathrm{d}v}
        {(1+v)^{\frac{p}2}}
    \leqslant \int_0^1 x^{-\frac12}(1-x)^{\frac{p-3}2}\,\mathrm{d}x
    = B\Big(\frac12,\frac{p-1}2\Big),\]
  where $B(\cdot,\cdot)$ is Euler's Beta function. 
  By denoting the above constant as $C$, we have $I_1\leqslant C(s+u_0)^{1-p}$. 
  Similary, $I_3\leqslant C(s-u_0)^{1-p}$ and 
  \[I_2= \int_{-s}^s \frac{\mathrm{d}u} {((u-u_0)^2+(t-v_0)^2)^{\frac{p}2}} 
       \leqslant \int_{\mathbb{R}} 
            \frac{\mathrm{d}u}{(v_0-t)^{p-1}(u^2+1)^{\frac{p}2}}
       \leqslant C(v_0-t)^{1-p},\]
  then 
  \[\int_{\Gamma{s,t}} |F(w)|^p |\mathrm{d}w|
    \leqslant C((s+u_0)^{1-p}+ (v_0-t)^{1-p}+ (s-u_0)^{1-p}),\]
  which means that $F(w)\in L^p(\Gamma_{s,t},|\mathrm{d}w|)$.
  
  If $w_0=u_0+\mathrm{i}v_0\in \mathbb{C}\setminus\overline{D_{s,t}}$, 
  then we choose $r>0$ big enough such that
  \[E=\Gamma_{s,t}\cap \{\mathrm{Im}\,w\leqslant |t|+|v_0|+1\}
    \subset D(w_0,r).\]
  Denote $d=\inf\{|w_0-w|\colon w\in\Gamma_{s,t}\}$, then $d>0$, 
  $|F(w)|\leqslant d^{-1}$ for $w\in \Gamma_{s,t}$ and 
  \begin{align*}
    \int_{\Gamma_{s,t}} |F(w)|^p |\mathrm{d}w|
    &= \bigg(\int_E+\int_{\Gamma_{s,t}\setminus E}\bigg) |F(w)|^p|\mathrm{d}w|\\
    &\leqslant d^{-p}\cdot 6r+ 2\int_{|t|+|v_0|+1}^{+\infty} 
          \frac{\mathrm{d}v}{(v-v_0)^p}                                 \\
    &= 6rd^{-p}+ \frac2{p-1}(|t|+|v_0|+1-v_0)^{1-p}                     \\
    &\leqslant 6rd^{-p}+\frac2{p-1}.
  \end{align*}
  Hence, we still have $F(w)\in L^p(\Gamma_{s,t},|\mathrm{d}w|)$.
\end{proof}

\begin{corollary}\label{cor-170826-0740}
  If $1< p\leqslant \infty$, $F(w)$ is a rational function which vanishes 
  at infinity and is with poles lying on $\Omega_-$, then 
  $F(w)\in H^p(\Omega_+)$. 
\end{corollary}

\begin{proof}
  Assume $1<p<\infty$, as the $p=\infty$ case is obvious. We consider 
  the simple case of $F(w)=\frac1{w-w_0}$ first, where $w_0\in\Omega_-$.
  Let $w_0=u_0+\mathrm{i}v_0$, $d=\inf\{|w_0-w|\colon w\in\Omega_+\}$, 
  then $d>0$. Choose $r= |w_0-\frac{\mathrm{i}}2(|v_0|+1)|
    + |\sigma-\frac{\mathrm{i}}2(|v_0|+1)|$, then
  $E= \Gamma\cap \{\mathrm{Im}\,w\leqslant |v_0|+1\}\subset D(w_0,r)$. 
  For $0<s<\sigma$, $t>0$, denote 
  $E_{s,t}= \Gamma_{s,t}\cap \{\mathrm{Im}\,w\leqslant |v_0|+1\}$, 
  then $E_{s,t}\subset E$ and by estimating as the second part in the proof 
  of Lemma~\ref{lem-170826-0720}, we have 
  \[m(s,t,F)= \bigg(\int_{\Gamma_{s,t}} |F(w)|^p |\mathrm{d}w|\bigg)^{\frac1p}
    \leqslant \Big(6rd^{-p}+\frac2{p-1}\Big)^{\frac1p}.\]
  Since the boundary is independent of $s$ and $t$, we know that 
  $F(w)\in H^p(\Omega_+)$.
  
  If $F(w)=\frac1{(w-w_0)^k}$ with $w_0\in\Omega_-$, where $k$ is 
  a positive integer, then $F(w)\in H^p(\Omega_+)$ since 
  $\frac1{(w-w_0)}\in H^{pk}(\Omega_+)$.
  
  For general $F(w)$, we could rewrite it as 
  \[F(w)= \sum_{j=1}^{N_1}\sum_{k=1}^{N_2} \frac{c_{jk}}{(w-w_j)^k}\]
  where $c_{jk}$'s are constants and $w_j\in\Omega_-$, then 
  $F(w)\in H^p(\Omega_+)$ follows from the linearity of $H^p(\Omega_+)$.
\end{proof}

\begin{corollary}
  If $1< p\leqslant \infty$, $F(w)$ is a rational function which vanishes 
  at infinity and is with poles lying on $\Omega_+$, then 
  $F(w)\in H^p(\Omega_-)$. 
\end{corollary}

\begin{proof}
  Let $1<p<\infty$ and $F(w)=\frac1{w-w_0}$ with 
  $w_0=u_0+\mathrm{i}v_0\in\Omega_+$, then $|u_0|<\sigma$, $v_0>0$. 
  For $s>\sigma$, $t<0$, by the first part in the proof of 
  Lemma~\ref{lem-170826-0720}, we have 
  \begin{align*}
    \int_{\Gamma_{s,t}} |F(w)|^p |\mathrm{d}w|
    &\leqslant C((s+u_0)^{1-p}+ (v_0-t)^{1-p}+ (s-u_0)^{1-p})         \\
    &\leqslant C(2(\sigma-|u_0|)^{1-p}+v_0^{1-p}),
  \end{align*}
  where $C= B(\frac12,\frac{p-1}2)$, then $F(w)\in H^p(\Omega_-)$.
  The rest cases are treated as in Corollary~\ref{cor-170826-0740}.
\end{proof}

Combing the above two corollaries, we know that $H^p(\Omega_\pm)$ is 
not empty for $1<p\leqslant\infty$. If $0<p\leqslant 1$, we choose 
positive integer $n$ such that $pn>1$, then 
$(w-w_0)^{-1}\in H^{pn}(\Omega_+)$ for $w_0\in\Omega_-$, and 
$(w-w_0)^{-n}\in H^p(\Omega_+)$. The same analysis applies to $H^p(\Omega_-)$
with $0<p\leqslant 1$.

\section{Boundedness of Cauchy Integral on $\Gamma$}

If $1\leqslant p<\infty$, $F(\zeta)\in L^p(\Gamma,|\mathrm{d}\zeta|)$, 
then Cauchy integral (or transform) of $F(\zeta)$ on $\Gamma$ is defined as
\[CF(w)= \frac1{2\pi\mathrm{i}} \int_{\Gamma} \frac{F(\zeta)}{\zeta-w}
      \mathrm{d}\zeta \quad\text{for } w\in\mathbb{C}\setminus\Gamma.\]
By H\"{o}lder's inequality and Lemma~\ref{lem-170826-0720}, $CF(w)$ is 
well-defined on $\mathbb{C}\setminus\Gamma$. In fact, it is also analytic.

\begin{lemma}\label{lem-170828-2000}
  If $1\leqslant p<\infty$, then $CF(w)$ is analytic on 
  $\mathbb{C}\setminus\Gamma$.
\end{lemma}

\begin{proof}
  Let $w$, $w_1\in\Omega_+$ with $w$ fixed, then
  \begin{align*}
    |CF(w)- CF(w_1)|
    &= \bigg|\frac1{2\pi\mathrm{i}} \int_{\Gamma} \bigg(\frac{F(\zeta)}{\zeta-w}
          - \frac{F(\zeta)}{\zeta-w_1}\bigg)\mathrm{d}\zeta\bigg|           \\
    &\leqslant \frac{|w-w_1|}{2\pi} \int_{\Gamma} 
          \frac{|F(\zeta)||\mathrm{d}\zeta|}{|\zeta-w||\zeta-w_1|},
  \end{align*}
  and we denote the last integral as $I$. Since $w\in\Omega_+$, there exists
  $\delta>0$, such that $D(w,2\delta)\subset\Omega_+$. For $\zeta\in\Gamma$
  and $w_1\in D(w,\delta)$, we have
  \[|w-w_1|< \delta< 2\delta\leqslant |\zeta-w|,\]
  then
  \[|\zeta- w_1|
    \geqslant |\zeta-w|- |w-w_1|
    \geqslant \frac12|\zeta-w|.\]
  It follows that, 
  \[|I|
    \leqslant \int_{\Gamma} \frac{2|F(\zeta)|}{|\zeta-w|^2}|\mathrm{d}\zeta|
    \leqslant 2\lVert F\rVert_{L^p(\Gamma,|\mathrm{d}\zeta|)}
        \lVert (\cdot- w)^{-1}\rVert_{L^{2q}(\Gamma,|\mathrm{d}\zeta|)}^2,\]
  where $\frac1p+\frac1q=1$ and 
  $(\zeta-w)^{-1}\in L^{2q}(\Gamma,|\mathrm{d}\zeta|)$
  by Lemma~\ref{lem-170826-0720}, since $1< q\leqslant \infty$. 
  We have proved that $I$ is bounded by 
  a constant which depends on $w$ only. Now let $w_1\to w$, then
  \[|CF(w)- CF(w_1)|\leqslant \frac{|w-w_1|}{2\pi} I\to 0,\]
  and $CF(w)$ is continuous on $\Omega_+$. It is then easy to verify, 
  by Morera's theorem, that $CF(w)$ is analytic on $\Omega_+$.
  
  We could prove that $CF(w)$ is analytic on $\Omega_-$ in the same way, 
  thus it is analytic on $\Omega_+\cup\Omega_-= \mathbb{C}\setminus\Gamma$.
\end{proof}

Actually, we could further prove that $CF(w)\in H^p(\Omega_\pm)$ 
for $1<p<\infty$ and the Cauchy transform is bounded, 
see Theorem~\ref{thm-170826-1440}. The following lemma 
has been proved in \cite{Dz66}, and is only a special case of a rather 
generalized theorem which has a long and complicated proof. The proof 
we provide here is greatly simplified, and is with a better transform norm, 
while the main idea still comes from the original one.

Remeber that, the Fourier transform of $f(t)\in L^2(\mathbb{R})$ is defined as
\[\hat{f}(x)= \frac1{\sqrt{2\pi}} \int_{\mathbb{R}} 
    f(t)\mathrm{e}^{-\mathrm{i}xt}\,\mathrm{d}t \quad
    \text{for } t\in\mathbb{R},\]
and, by Plancherel theorem, $\lVert \hat{f}\rVert_{L^2(\mathbb{R})}
  = \lVert f\rVert_{L^2(\mathbb{R})}$, while Parseval formula shows that 
\[\int_{\mathbb{R}} f(t)\overline{g(t)}\,\mathrm{d}t
  = \int_{\mathbb{R}} \hat{f}(x)\overline{\hat{g}(x)}\,\mathrm{d}x,\]
for $f(t)$, $g(t)\in L^2(\mathbb{R})$. See~\cite{Ru87}.

\begin{lemma}[\cite{Dz66}]\label{lem-170826-1020}
  If $f(t)\in L^2(\mathbb{R}_+)$, and define
  \[g(y)= \int_{\mathbb{R}_+} \mathrm{e}^{-yt} f(t)\,\mathrm{d}t\quad
    \text{for } y>0,\]
  then $\lVert g\rVert_{L^2(\mathbb{R}_+)}
    \leqslant \sqrt{\pi} \lVert f\rVert_{L^2(\mathbb{R}_+)}$.
\end{lemma}

\begin{proof}
  Replace $t$ with $\mathrm{e}^t$ in the above integral,
  \[g(y)= \int_{\mathbb{R}} \mathrm{e}^{-y\mathrm{e}^t} f(\mathrm{e}^t)
      \mathrm{e}^t\,\mathrm{d}t
    = \int_{\mathbb{R}} f(\mathrm{e}^t)\mathrm{e}^{\frac{t}2}\cdot
         \overline{\mathrm{e}^{-y\mathrm{e}^t+\frac{t}2}}\,\mathrm{d}t.\]
  Since
  \[\int_{\mathbb{R}} |f(\mathrm{e}^t)\mathrm{e}^{\frac{t}2}|^2\,\mathrm{d}t
    = \int_{\mathbb{R}_+} |f(t)|^2\,\mathrm{d}t
    = \lVert f\rVert_{L^2(\mathbb{R}_+)}^2,\]
  or $\lVert f(\mathrm{e}^t)\mathrm{e}^{\frac{t}2}\rVert_{L^2(\mathbb{R})}
    = \lVert f\rVert_{L^2(\mathbb{R}_+)}$, and for $y>0$ fixed,
  \[\int_{\mathbb{R}} |\mathrm{e}^{-y\mathrm{e}^t+\frac{t}2}|^2\,\mathrm{d}t
    = \int_{\mathbb{R}_+} \mathrm{e}^{-2yt}\,\mathrm{d}t
    = \frac1{2y},\]
  then both $f(\mathrm{e}^t)\mathrm{e}^{\frac{t}2}$ and 
  $\mathrm{e}^{-y\mathrm{e}^t+\frac{t}2}$ are in $L^2(\mathbb{R})$, 
  and the Fourier transform of the latter is
  \[(\mathrm{e}^{-y\mathrm{e}^t+\frac{t}2})\,\hat{}\,(x)
    = \frac1{\sqrt{2\pi}} \int_{\mathbb{R}} 
        \mathrm{e}^{-y\mathrm{e}^t+\frac{t}2}
        \mathrm{e}^{-\mathrm{i}xt}\,\mathrm{d}t
    = \frac1{\sqrt{2\pi}} \int_{\mathbb{R}_+} 
        \mathrm{e}^{-yt} t^{-\frac12-\mathrm{i}x}\,\mathrm{d}t.\]
  Denote $(f(\mathrm{e}^t)\mathrm{e}^{\frac{t}2})\,\hat{}\,(x)$ as $h(x)$, then
  \[\lVert h\rVert_{L^2(\mathbb{R})}
    = \lVert f(\mathrm{e}^t)\mathrm{e}^{\frac{t}2}\rVert_{L^2(\mathbb{R})}
    = \lVert f\rVert_{L^2(\mathbb{R}_+)},\]
  and, by Parseval formula,
  \begin{align*}
    g(y)
    &= \int_{\mathbb{R}} (f(\mathrm{e}^t)\mathrm{e}^{\frac{t}2})\,\hat{}\,(x) 
           \overline{(\mathrm{e}^{-y\mathrm{e}^t+\frac{t}2})
              \,\hat{}\,(x)}\,\mathrm{d}x                  \\
    &= \int_{\mathbb{R}} h(x) \frac1{\sqrt{2\pi}} \int_{\mathbb{R}_+} 
          \mathrm{e}^{-yt} t^{-\frac12+\mathrm{i}x}\,\mathrm{d}t\,\mathrm{d}x   
  \end{align*}
  After replacing $t$ with $\frac{t}y$, we have 
  \begin{align*}
    g(y)
    &= \frac1{\sqrt{2\pi}} \int_{\mathbb{R}} h(x) \int_{\mathbb{R}_+} 
          \mathrm{e}^{-t} \Big(\frac{t}y\Big)^{-\frac12+\mathrm{i}x} 
          \frac{\mathrm{d}t}{y}\,\mathrm{d}x                        \\
    &= \frac1{\sqrt{2\pi}} \int_{\mathbb{R}} h(x) \int_{\mathbb{R}_+} 
          \mathrm{e}^{-t} t^{-\frac12+\mathrm{i}x}\,\mathrm{d}t 
          \,y^{-\frac12-\mathrm{i}x}\,\mathrm{d}x.
  \end{align*}
  
  Define $h_1(x)= h(x) \int_{\mathbb{R}_+} 
    \mathrm{e}^{-t} t^{-\frac12+\mathrm{i}x}\,\mathrm{d}t$, then
  \[g(y)
    = \frac1{\sqrt{2\pi}} \int_{\mathbb{R}} h_1(x) 
         y^{-\frac12-\mathrm{i}x}\,\mathrm{d}x,\quad
    \text{or\quad}
    g(\mathrm{e}^y)\mathrm{e}^{\frac{y}2}
    = \frac1{\sqrt{2\pi}} \int_{\mathbb{R}} h_1(x) 
         \mathrm{e}^{-\mathrm{i}yx}\,\mathrm{d}x.\]
  Since $h(x)\in L^2(\mathbb{R})$, and
  \[\bigg|\int_{\mathbb{R}_+} \mathrm{e}^{-t} 
        t^{-\frac12+\mathrm{i}x}\,\mathrm{d}t\bigg|
    \leqslant \int_{\mathbb{R}_+} t^{-\frac12}
        \mathrm{e}^{-t} \,\mathrm{d}t
    = \Gamma\Big(\frac12\Big)
    = \sqrt{\pi},\]
  where $\Gamma(\cdot)$ is Eular's Gamma function, then 
  $\lVert h_1\rVert_{L^2(\mathbb{R})}
    \leqslant \sqrt{\pi}\lVert h\rVert_{L^2(\mathbb{R})}$, and it follows that,
  $g(\mathrm{e}^y)\mathrm{e}^{\frac{y}2}= \widehat{h_1}(y)$, and 
  \begin{align*}
    \int_{\mathbb{R}} |g(\mathrm{e}^y)\mathrm{e}^{\frac{y}2}|^2\mathrm{d}y
    &= \lVert \widehat{h_1}\rVert_{L^2(\mathbb{R})}^2
     = \lVert h_1\rVert_{L^2(\mathbb{R})}^2                   \\
    &\leqslant \pi\lVert h\rVert_{L^2(\mathbb{R})}^2
     = \pi\lVert f\rVert_{L^2(\mathbb{R}_+)}^2.
  \end{align*}
  The left side above is obviously $\lVert g\rVert_{L^2(\mathbb{R}_+)}^2$, 
  thus $\lVert g\rVert_{L^2(\mathbb{R}_+)}
    \leqslant \sqrt{\pi} \lVert f\rVert_{L^2(\mathbb{R}_+)}$.
\end{proof}

The next corollary of Lemma~\ref{lem-170826-1020} is crucial to 
our proof of the boundedness of Cauchy transform of 
$L^p(\Gamma,|\mathrm{d}\zeta|)$ ($1<p<\infty$) functions, 
and we need a fatorization lemma on $H^p(\mathbb{C}_+)$ ($0<p<\infty$) 
during its proof.

\begin{lemma}[\cite{Ga07}]\label{lem-170801-1530}
  Let $\{z_n= x_n+\mathrm{i}y_n\}$ be a sequence of points in $\mathbb{C}_+$, 
  such that 
  \[\sum_{n=1}^{\infty}\frac{y_n}{1+|z_n|^2}< \infty,\]
  and $m$ be the number of $z_n$ equal to $\mathrm{i}$. 
  Then the Blaschke product
  \[B(z)= \bigg(\frac{z-\mathrm{i}}{z+\mathrm{i}}\bigg)^m
      \prod_{z_n\neq\mathrm{i}} \frac{|z_n^2+1|}{z_n^2+1}
      \cdot \frac{z-z_n}{z-\overline{z_n}}\]
  converges on $\mathbb{C}_+$, has non-tangential boundary limit $B(x)$ a.e.\@ 
  on $\mathbb{R}$, and the zeros of $B(z)$ are precisely 
  the points $z_n$, both counting multiplicity. Moreover, $|B(z)|<1$ 
  on $\mathbb{C}_+$ and $|B(x)|=1$ a.e.\@ on $\mathbb{R}$.
\end{lemma}

\begin{lemma}[\cite{Ga07}]\label{lem-170801-1550}
  If $0<p<\infty$, $f(z)\in H^p(\mathbb{C}_+)$, $f\not\equiv 0$, 
  and $B(z)$ is the Blaschke product associated with the zeros of $f(z)$, 
  Then
  \[g(z)=\frac{f(z)}{B(z)}\neq 0,\quad
    \text{and } \lVert g\rVert_{H^p(\mathbb{C}_+)}
    = \lVert f\rVert_{H^p(\mathbb{C}_+)}.\]
\end{lemma}

\begin{corollary}\label{cor-170826-1220}
  If $0<p<\infty$, $f(z)\in H^p(\mathbb{C}_+)$, $y>0$ and $x\in\mathbb{R}$, 
  then there exists a positive function $g(\mathrm{i}y)$ on $\mathbb{R}_+$,
  such that $|f(x+\mathrm{i}y)|\leqslant g(\mathrm{i}y)$ for all 
  $x\in\mathbb{R}$, and
  \[\lVert f(x+\mathrm{i}\cdot)\rVert_{L^p(\mathbb{R}_+)}
    \leqslant \lVert g(\mathrm{i}\cdot)\rVert_{L^p(\mathbb{R}_+)}
    \leqslant 2^{-\frac1p} \lVert f\rVert_{H^p(\mathbb{C}_+)}.\]
\end{corollary}

\begin{proof}
  We consider the $p=2$ case first, then by one of Paley-Wiener 
  theorems~\cite{Ru87}, there exists $h(t)\in L^2(\mathbb{R}_+)$, such that
  \[f(z)= \frac1{\sqrt{2\pi}} \int_{\mathbb{R}_+} 
      h(t)\mathrm{e}^{\mathrm{i}tz}\,\mathrm{d}t,\]
  with $\lVert f\rVert_{H^2(\mathbb{C}_+)}
    = \lVert h\rVert_{L^2(\mathbb{R}_+)}$, then
  \[|f(x+\mathrm{i}y)|
    = \frac1{\sqrt{2\pi}} \bigg|\int_{\mathbb{R}_+} 
        h(t)\mathrm{e}^{\mathrm{i}t(x+\mathrm{i}y)}\,\mathrm{d}t\bigg|
    \leqslant \frac1{\sqrt{2\pi}}\int_{\mathbb{R}_+} 
        |h(t)|\mathrm{e}^{-ty}\,\mathrm{d}t,\] 
  and we denote the last expression as $g_1(\mathrm{i}y)$. Since $|h(t)|$ 
  and $h(t)$ have the same $L^2(\mathbb{R}_+)$ norm, 
  we have, by Lemma~\ref{lem-170826-1020}, 
  \[\lVert f(x+\mathrm{i}\cdot)\rVert_{L^2(\mathbb{R}_+)}
    \leqslant \lVert g_1(\mathrm{i}\cdot)\rVert_{L^2(\mathbb{R}_+)}
    \leqslant \frac1{\sqrt2} \lVert h\rVert_{L^2(\mathbb{R}_+)}
    = \frac1{\sqrt2} \lVert f\rVert_{H^2(\mathbb{C}_+)}.\]
  
  For other $p\in(0,\infty)$, we could write $f(z)= B(z)Q(z)$ 
  where $Q(z)\neq0$ with $\lVert Q\rVert_{H^p(\mathbb{C}_+)}
    = \lVert f\rVert_{H^p(\mathbb{C}_+)}$, and $B(z)$ is the Blaschke product.
  Then $Q^{\frac{p}2}(z)\in H^2(\mathbb{C}_+)$ and by what we have proved, 
  there exists a positive function $g_2(\mathrm{i}y)$ such that
  $|Q^{\frac{p}2}(x+\mathrm{i}y)|\leqslant g_2(\mathrm{i}y)$ 
  for all $x\in\mathbb{R}$, and
  \[\lVert Q^{\frac{p}2}(x+\mathrm{i}\cdot)\rVert_{L^2(\mathbb{R}_+)}
    \leqslant \lVert g_2(\mathrm{i}\cdot)\rVert_{L^2(\mathbb{R}_+)}
    \leqslant \frac1{\sqrt2}\lVert Q^{\frac{p}2}\rVert_{H^2(\mathbb{C}_+)},\]
  or
  \[\lVert Q(x+\mathrm{i}\cdot)\rVert_{L^p(\mathbb{R}_+)}
    \leqslant \lVert g_2^{\frac2p}(\mathrm{i}\cdot)\rVert_{L^p(\mathbb{R}_+)}
    \leqslant 2^{-\frac1p}\lVert Q\rVert_{H^p(\mathbb{C}_+)}.\]
  It follows that $|f(x+\mathrm{i}y)|\leqslant |Q(x+\mathrm{i}y)|
    \leqslant g_2^{\frac2p}(\mathrm{i}y)$, and
  \[\lVert f(x+\mathrm{i}\cdot)\rVert_{L^p(\mathbb{R}_+)}
    \leqslant \lVert g_2^{\frac2p}(\mathrm{i}\cdot)\rVert_{L^p(\mathbb{R}_+)}
    \leqslant 2^{-\frac1p}\lVert f\rVert_{H^p(\mathbb{C}_+)}.\]
  Denote $g_2^{\frac2p}(\mathrm{i}y)$ as $g(\mathrm{i}y)$, 
  then the proof is finished.
\end{proof}

Hardy space $H^p(D)$ for $0< p\leqslant \infty$, where $D$ is the 
translation and rotation of $\mathbb{C}_+$, is defined similarly as that of 
$H^p(\mathbb{C}_+)$. For example, if $0< p< \infty$, then 
$H^p(\{\mathrm{Re}\,w> -\sigma\})$ are analytic functions equipped 
with $H^p$-norm
\[\lVert F\rVert_{H^p(\{\mathrm{Re}\,w> -\sigma\})}
  = \sup_{u>-\sigma} \bigg(\int_{\mathbb{R}} |F(u+\mathrm{i}v)|^p
      \,\mathrm{d}v\bigg)^{\frac1p}.\] 

\begin{proposition}[\cite{Vi94}]\label{pro-170826-1200}
  If $1\leqslant p\leqslant \infty$, $\sigma>0$, 
  $F_1(w)\in H^p(\{\mathrm{Re}\,w> -\sigma\})$,
  $F_2(w)\in H^p(\mathbb{C}_+)$ and
  $F_3(w)\in H^p(\{\mathrm{Re}\,w< \sigma\})$. 
  Define $F(w)=F_1(w)+F_2(w)+F_3(w)$ for $\Omega_+= D_{\sigma,0}$, then 
  $F(w)\in H^p(\Omega_+)$. We may simply write
  \[H^p(\{\mathrm{Re}\,w> -\sigma\})+ H^p(\mathbb{C}_+)
    + H^p(\{\mathrm{Re}\,w< \sigma\})\subset H^p(\Omega_+).\]
\end{proposition}

\begin{proof}
  Let $1\leqslant p<\infty$, $0<s<\sigma$ and $t>0$, then
  \begin{align*}
    m(s,t,F)
    &= \bigg(\int_{\Gamma{s,t}} |F(w)|^p |\mathrm{d}w|\bigg)^{\frac1p}  \\
    &= \bigg(\int_{\Gamma{s,t}} \Big|\sum_{j=1}^3 F_j(w)\Big|^p 
         |\mathrm{d}w|\bigg)^{\frac1p}
     \leqslant \sum_{j=1}^3 \bigg(\int_{\Gamma{s,t}} |F_j(w)|^p 
         |\mathrm{d}w|\bigg)^{\frac1p}.
  \end{align*}
  By the definition of $H^p(\{\mathrm{Re}\,w> -\sigma\})$ and 
  Corollary~\ref{cor-170826-1220},
  \begin{align*}
    \int_{\Gamma{s,t}} |F_1(w)|^p |\mathrm{d}w|
    &= \sum_{k=1}^3 \int_{\Gamma{s,t,k}} |F_1(w)|^p |\mathrm{d}w|  \\
    &\leqslant \Big(1+\frac12+1\Big) 
        \lVert F_1\rVert_{H^p(\{\mathrm{Re}\,w> -\sigma\})}^p
    = \frac52 \lVert F_1\rVert_{H^p(\{\mathrm{Re}\,w> -\sigma\})}^p.
  \end{align*}
  Similarly, we have
  \begin{align*}
    \int_{\Gamma{s,t}} |F_2(w)|^p |\mathrm{d}w|
    &\leqslant 2 \lVert F_2\rVert_{H^p(\mathbb{C}_+)}^p,             \\
    \int_{\Gamma{s,t}} |F_3(w)|^p |\mathrm{d}w|
    &\leqslant \frac52 \lVert F_3\rVert_{H^p(\{\mathrm{Re}\,w<\sigma\})}^p,
  \end{align*}
  then
  \[m(s,t,F)
    \leqslant \Big(\frac52\Big)^{\frac1p} 
        \lVert F_1\rVert_{H^p(\{\mathrm{Re}\,w> -\sigma\})}
      + 2^{\frac1p} \lVert F_2\rVert_{H^p(\mathbb{C}_+)}
      + \Big(\frac52\Big)^{\frac1p} 
        \lVert F_3\rVert_{H^p(\{\mathrm{Re}\,w< \sigma\})},\]
  which means that $F(w)\in H^p(\Omega_+)$.
\end{proof}

The converse of Proposition~\ref{pro-170826-1200} will be proved 
in Theorem~\ref{thm-170828-1040}, and the $H^p(\Omega_-)$ version is 
much easier to prove by invoking definitions.

\begin{theorem}[\cite{Vi94}]\label{thm-170826-1410}
  If $0<p\leqslant\infty$ and $F(w)$ is analytic on $\Omega_-$,
  then $F(w)\in H^p(\Omega_-)$  if and only if 
  $F(w)$ is in $H^p(\{\mathrm{Re}\,w< -\sigma\})$, $H^p(\mathbb{C}_-)$, and
  $H^p(\{\mathrm{Re}\,w> \sigma\})$.
\end{theorem}

\begin{proof}
  We only prove the ``only if'' part with $0<p<\infty$, as the other parts
  is obvious by definition. For any $s>\sigma$ and $t<0$, we have
  \[\int_{-s}^s |F(u+\mathrm{i}t)|^p\,\mathrm{d}u
    = \int_{\Gamma_{s,t,2}} |F(w)|^p|\mathrm{d}w|
    \leqslant \lVert F\rVert_{H^p(\Omega_-)}^p,\]
  then, by Fatou's lemma,
  \begin{align*}
    \int_{\mathbb{R}} |F(u+\mathrm{i}t)|^p\,\mathrm{d}u
    &= \int_{\mathbb{R}} \liminf_{s\to\infty} \chi_{[-s,s]} 
          |F(u+\mathrm{i}t)|^p\,\mathrm{d}u                  \\
    &\leqslant \liminf_{s\to\infty} \int_{\mathbb{R}} \chi_{[-s,s]} 
          |F(u+\mathrm{i}t)|^p\,\mathrm{d}u
    \leqslant \lVert F\rVert_{H^p(\Omega_-)}^p.
  \end{align*}
  Hence, 
  \[\lVert F\rVert_{H^p(\mathbb{C}_-)}^p
    = \sup_{t<0} \int_{\mathbb{R}} |F(u+\mathrm{i}t)|^p\,\mathrm{d}u
    \leqslant \lVert F\rVert_{H^p(\Omega_-)}^p,\]
  and $F(w)\in H^p(\mathbb{C}_-)$. The other two inclusions could be 
  similarly verified.
\end{proof}

Before proving that Cauchy transform is bounded on  
$L^p(\Gamma,|\mathrm{d}\zeta|)$ ($1<p<\infty$), we introduce the boundedness 
of Cauchy transform on $L^p(\mathbb{R})$ ($1<p<\infty$).

\begin{lemma}[\cite{De10}]\label{lem-170826-1500}
  Suppose $1<p<\infty$, $f(t)\in L^p(\mathbb{R})$, and define 
  \[Cf(z)= \frac1{2\pi\mathrm{i}} \int_{\mathbb{R}} 
      \frac{f(t)}{t-z}\,\mathrm{d}t\quad
    \text{for } z\neq \mathbb{R},\]
  then 
  \[\sup_{y>0}\bigg(\int_{\mathbb{R}} |Cf(x+\mathrm{i}y)|^p
      \,\mathrm{d}x\bigg)^{\frac1p}
    \leqslant A_p \lVert f\rVert_{L^p(\mathbb{R})},\]
  where $A_p^p= \max\{\frac{p}{p-1}, p^{p-1}\}$.
\end{lemma}

The above lemma clearly implies that $Cf(z)\in H^p(\mathbb{C}_\pm)$ 
for $1<p<\infty$, since it is easy to verify that $Cf(z)$ is analytic on 
$\mathbb{C}\setminus \mathbb{R}$. Also, the transform norm do not exceed $A_p$.

\begin{theorem}\label{thm-170826-1440}
  If $1<p<\infty$, $F(\zeta)\in L^p(\Gamma,|\mathrm{d}\zeta|)$ and 
  $CF(w)$ is the Cauchy integral of $F(\zeta)$ on $\Gamma$ with 
  $w\in\Omega_\pm$, then $CF(w)\in H^p(\Omega_\pm)$, and 
  \begin{align*}
    \lVert CF\rVert_{H^p(\Omega_+)}
    &\leqslant \Big(\frac52\Big)^{\frac1p} A_p
        \lVert F\rVert_{L^p(\Gamma,|\mathrm{d}\zeta|)},          \\
    \lVert CF\rVert_{H^p(\Omega_-)}
    &\leqslant 3^{\frac1p} A_p
        \lVert F\rVert_{L^p(\Gamma,|\mathrm{d}\zeta|)},   
  \end{align*}
  where $A_p^p= \max\{\frac{p}{p-1}, p^{p-1}\}$.
\end{theorem}

\begin{proof}
  We have already proved that $CF(w)$ is analytic on $\Omega_+\cup\Omega_-$
  in Lemma~\ref{lem-170828-2000}, thus only need to verify the bounded 
  integrability in definition of $H^p$ spaces. Let 
  $\gamma_1= \{\mathrm{Re}\,w= -\sigma\}$, 
  $\gamma_2=\mathbb{R}$ and $\gamma_3= \{\mathrm{Re}\,w= \sigma\}$, then
  \[CF(w)
    = \frac1{2\pi\mathrm{i}} \int_{\Gamma} 
        \frac{F(\zeta)}{\zeta-w}\,\mathrm{d}\zeta
    = \sum_{j=1}^3 \frac1{2\pi\mathrm{i}} \int_{\gamma_j} 
        \frac{\chi_{\Gamma_j}F(\zeta)}{\zeta-w}\,\mathrm{d}\zeta
    = \sum_{j=1}^3 G_j(w),\]
  and $G_j(w)$ is well-defined on $\mathbb{C}\setminus \overline{\Gamma_j}$
  for $j=1$, $2$, $3$.
  
  If $w\in\Omega_+$, then $CF(w)$ is the sum of 
  $G_1(w)\in H^p(\{\mathrm{Re}\,w> -\sigma\})$,
  $G_2(w)\in H^p(\mathbb{C}_+)$, 
  and $G_3(w)\in H^p(\{\mathrm{Re}\,w< \sigma\})$. Let $0<s<\sigma$, $t>0$, 
  then by Proposition~\ref{pro-170826-1200} and Lemma~\ref{lem-170826-1500},
  \begin{align*}
    m(s,t,CF)
    &= \bigg(\int_{\Gamma_{s,t}} |CF(w)|^p |\mathrm{d}w|\bigg)^{\frac1p}  \\
    &\leqslant \Big(\frac52\Big)^{\frac1p} 
          \lVert G_1\rVert_{H^p(\{\mathrm{Re}\,w> -\sigma\})}
        + 2^{\frac1p} \lVert G_2\rVert_{H^p(\mathbb{C}_+)}
        + \Big(\frac52\Big)^{\frac1p} 
          \lVert G_3\rVert_{H^p(\{\mathrm{Re}\,w< \sigma\})}         \\
    &\leqslant \Big(\frac52\Big)^{\frac1p}\sum_{j=1}^3 A_p 
        \lVert \chi_{\Gamma_j}F \rVert_{L^p(\gamma_j,|\mathrm{d}\zeta|)}  \\
    &= \Big(\frac52\Big)^{\frac1p} A_p 
            \lVert F \rVert_{L^p(\Gamma,|\mathrm{d}\zeta|)},
  \end{align*}
  and it shows that $F(w)\in H^p(\Omega_+)$ with
  \[\lVert CF\rVert_{H^p(\Omega_+)}
    \leqslant \Big(\frac52\Big)^{\frac1p} A_p
        \lVert F\rVert_{L^p(\Gamma,|\mathrm{d}\zeta|)}.\]
  
  If $w\in\Omega_-$, then $CF(w)$ is the sum of three $H^p$ functions $G_j(w)$
  for $j=1$, $2$, $3$, where
  
  $G_1(w)\in H^p(\{\mathrm{Re}\,w> -\sigma\})$ 
  or $H^p(\{\mathrm{Re}\,w< -\sigma\})$,
  
  $G_2(w)\in H^p(\mathbb{C}_+)$ or $H^p(\mathbb{C}_-)$, and
  
  $G_3(w)\in H^p(\{\mathrm{Re}\,w< \sigma\})$
  or $H^p(\{\mathrm{Re}\,w> \sigma\})$, \\
  depending on the location of $w$. 
  Let $s>\sigma$, $t<0$, then by Lemma~\ref{lem-170826-1500}, definitions of
  $H^p(\{\mathrm{Re}\,w> -\sigma\})$ and $H^p(\{\mathrm{Re}\,w< -\sigma\})$,
  and Corollary~\ref{cor-170826-1220}, 
  \begin{align*}
    \int_{\Gamma_{s,t}} |G_1(w)|^p |\mathrm{d}w|
    &= \sum_{k=1}^3 \int_{\Gamma_{s,t,k}} |G_1(w)|^p |\mathrm{d}w|  \\
    &\leqslant \Big(1+\frac12\Big) 
          \lVert G_1\rVert_{H^p(\{\mathrm{Re}\,w< -\sigma\})}^p
        + \Big(\frac12+1\Big) 
          \lVert G_1\rVert_{H^p(\{\mathrm{Re}\,w> -\sigma\})}^p     \\
    &\leqslant 3A_p^p \lVert \chi_{\Gamma_1}F 
        \rVert_{L^p(\gamma_1,|\mathrm{d}\zeta|)}^p.
  \end{align*}
  Similarly, we have
  \[\int_{\Gamma_{s,t}} |G_j(w)|^p |\mathrm{d}w|
    \leqslant 3A_p^p \lVert \chi_{\Gamma_j}F 
         \rVert_{L^p(\gamma_j,|\mathrm{d}\zeta|)}^p\]
  for $j=2$, $3$, then 
  \begin{align*}
    m(s,t,CF)
    &\leqslant \sum_{j=1}^3 \bigg(\int_{\Gamma_{s,t}} |G_j(w)|^p 
           |\mathrm{d}w|\bigg)^{\frac1p}                       \\
    &\leqslant 3^{\frac1p}A_p \sum_{j=1}^3 \lVert F 
           \rVert_{L^p(\Gamma_j,|\mathrm{d}\zeta|)}
     = 3^{\frac1p}A_p \lVert F\rVert_{L^p(\Gamma,|\mathrm{d}\zeta|)},
  \end{align*}
  which implies that $F(w)\in H^p(\Omega_-)$ and
  \[\lVert CF\rVert_{H^p(\Omega_-)}
    \leqslant 3^{\frac1p} A_p\lVert F\rVert_{L^p(\Gamma,|\mathrm{d}\zeta|)}.\]
  The proof of this theorem is thus finished.
\end{proof}

The lines $\gamma_1= \{\mathrm{Re}\,w= -\sigma\}$, $\gamma_2=\mathbb{R}$, 
$\gamma_3= \{\mathrm{Re}\,w= \sigma\}$ introduced in the proof of 
Theorem~\ref{thm-170826-1440} are also important for proving some of the
following results. Let $\gamma= \gamma_1\cup \gamma_2\cup \gamma_3$, 
and for $s>0$, $t\in\mathbb{R}$, let 
$\gamma_{s,t,1}= \gamma_1\cap \{\mathrm{Im}\,w>t\}$, 
$\gamma_{s,t,2}= \gamma_2\cap \{|\mathrm{Re}\,w|\leqslant s\}$, 
$\gamma_{s,t,3}= \gamma_3\cap \{\mathrm{Im}\,w>t\}$, 
and $\gamma_{s,t}= \gamma_{s,t,1}\cup \gamma_{s,t,2}\cup \gamma_{s,t,3}$. 
The orientation of $\gamma_1$ is from top to bottom, that of $\gamma_2$ from 
left to right, and that of $\gamma_3$ from bottom to top, 
then $\Gamma=\gamma_{\sigma,0}$ with the same orientation.

We now define a one-to-one mapping~$P_{s,t}$ from $\gamma_{s,t}$ onto 
$\Gamma_{s,t}$. For $\zeta\in\gamma_{s,t}$, define
\[\zeta_{s,t}= P_{s,t}(\zeta)= \left\{\!\!
        \begin{array}{ll}
          \zeta+(\sigma-s)  & \text{if $\zeta\in \gamma_{s,t,1}$},\\
          \zeta+\mathrm{i}t & \text{if $\zeta\in \gamma_{s,t,2}$},\\
          \zeta-(\sigma-s)  & \text{if $\zeta\in \gamma_{s,t,3}$},
        \end{array}\right.\]
then the inverse mapping $P_{s,t}^{-1}$ is 
\[\zeta= P_{s,t}^{-1}(w)= \left\{\!\!
        \begin{array}{ll}
          w-(\sigma-s)  & \text{if $w\in \Gamma_{s,t,1}$},\\
          w-\mathrm{i}t & \text{if $w\in \Gamma_{s,t,2}$},\\
          w+(\sigma-s)  & \text{if $w\in \Gamma_{s,t,3}$},
        \end{array}\right.\]
and $\zeta\in \gamma_{s,t,j}$ if and only if $\zeta_{s,t}\in \Gamma_{s,t,j}$
for $j=1$, $2$, $3$. In fact, $P_{s,t}$ and $P_{s,t}^{-1}$ are just 
combinations of translation.

Then for $G(w)$ defined on $\Gamma_{s,t}$, we may view it as a function
$G_{s,t}(\zeta)$ defined on $\gamma$, that is, we let
\[G_{s,t}(\zeta)= \left\{\!\!
      \begin{array}{ll}
        G(\zeta_{s,t})= G(P_{s,t}(\zeta)) & \text{for $\zeta\in\gamma_{s,t}$},\\
        0 & \text{for $\zeta\in \gamma\setminus\gamma_{s,t}$}.
      \end{array}\right.\]
Obviously, $G_{s,t}(\zeta)= \chi_{\gamma_{s,t}}G_{s,t}(\zeta)$,
\[\int_\gamma G_{s,t}(\zeta)\,\mathrm{d}\zeta
  = \int_{\gamma_{s,t}} G_{s,t}(\zeta)\,\mathrm{d}\zeta           
  = \int_{\Gamma_{s,t}} G(\zeta_{s,t})\,\mathrm{d}\zeta_{s,t}
  = \int_{\Gamma_{s,t}} G(w)\,\mathrm{d}w,\]
and, similarly,
\[\int_\gamma G_{s,t}(\zeta)|\mathrm{d}\zeta|
  = \int_{\Gamma_{s,t}} G(w)|\mathrm{d}w|.\]
If $0< s\leqslant \sigma$, $t\geqslant 0$, then 
$\gamma_{s,t}\subset \gamma_{\sigma,0}= \Gamma$, and $G_{s,t}(\zeta)$ 
could be considered as a function only defined on $\Gamma$.

\begin{lemma}\label{lem-170826-1650}
  If $1<p<\infty$ and $f(z)\in H^p(\mathbb{C}_+)$, then for $y>0$, 
  $|f(x+\mathrm{i}y)|$ is dominated by $\frac{10}{\pi}f^*(x)\in L^p(\mathbb{R})$,
  where $f^*(x)$ is the Hardy-Littlewood maximal function of $f(x)$, 
  the non-tangential boundary limit of $f(z)$.
\end{lemma}

The proof of the above lemma is outlined in~\cite{MC97}, which involves
utilizing the Poisson representation of $f(z)$ by $f(x)$ 
and dividing $\mathbb{R}$ properly. We have the following domination theorem
on $\Gamma_{s,t}$.

\begin{theorem}\label{thm-170826-1700}
  If $1<p<\infty$, $s\in\mathbb{R}_+\setminus\{\sigma\}$, 
  $t\in\mathbb{R}\setminus\{0\}$, 
  and $F(\zeta)\in L^p(\Gamma,|\mathrm{d}\zeta|)$, 
  then $|(CF)_{s,t}(\zeta)|$ is dominated by a function 
  $g(\zeta)\in L^p(\gamma,|\mathrm{d}\zeta|)$, 
  where $\zeta\in\gamma\setminus\{\pm\sigma\}$ and $CF(w)$ is 
  the Cauchy integral of $F(\zeta)$ on $\Gamma$.
\end{theorem}

\begin{proof}
  We write, by definition of $CF(w)$, for $w\in \Omega_+\cup\Omega_-$,
  \[CF(w)
    = \sum_{j=1}^3 \frac1{2\pi\mathrm{i}} \int_{\gamma_j} 
        \frac{\chi_{\Gamma_j}F(\zeta)}{\zeta-w}\,\mathrm{d}\zeta
    = \sum_{j=1}^3 G_j(w),\]
  then $G_j(w)$'s are $H^p$ functions on corresponding domains, and their
  non-tangential boundary limit functions are denoted as $g_{j\pm}(\zeta)$ 
  with $\zeta\in\gamma_j$ for $j=1$, $2$, $3$. Here the signs in subsripts 
  depend on ``left'' or ``right'' of the domains relative to their boundaries. 
  For example, $\{\mathrm{Re}\,w>-\sigma\}$ is on the left of $\gamma_1$, 
  then $G_1(w)$ with $\mathrm{Re}\,w>-\sigma$ has non-tangential boundary limit 
  $g_{1+}(\zeta)$, while $g_{1-}(\zeta)$ is the non-tangential boundary limit 
  of $G_1(w)$ with $\mathrm{Re}\,w<-\sigma$. The other $g_{j\pm}(\zeta)$'s
  are defined accordingly.
  
  Then, by Lemma~\ref{lem-170826-1650}, $|G_j(P_{s,t}(\zeta))|
    \leqslant \frac{10}{\pi} g_{j\pm}^*(\zeta)$ where $\zeta\in\gamma_{s,t,j}$
  for $j=1$, $2$, $3$, with signs depending on where $P_{s,t}(\zeta)$ locates
  and $g_{j\pm}^*(\zeta)$'s are the Hardy-Littlewood maximal functions.
  By Corollary~\ref{cor-170826-1220}, there exists 
  $h_{j\pm}(\zeta)$ for $j=1$, $2$, $3$, where 
  
  $h_{1+}$ is defined on $\gamma_2\cap\{\mathrm{Re}\,w> -\sigma\}$, 
  $h_{1-}$ on $\gamma_2\cap\{\mathrm{Re}\,w< -\sigma\}$;
  
  $h_{2+}$ on $\{\mathrm{i}v\colon v>0\}$, 
  $h_{2-}$ on $\{\mathrm{i}v\colon v<0\}$;
  
  $h_{3+}$ on $\gamma_2\cap\{\mathrm{Re}\,w< \sigma\}$, 
  $h_{3-}$ on $\gamma_2\cap\{\mathrm{Re}\,w> \sigma\}$,\\
  such that
  $|G_j(w)|\leqslant h_{j\pm}(\mathrm{Re}\,w)$ for $j=1$, $3$, 
  $w\notin\gamma_1\cup\gamma_3$, and 
  $|G_2(w)|\leqslant h_{2\pm}(\mathrm{Im}\,w)$ for $w\notin\gamma_2$.
  Besides, $\lVert h_{j\pm}\rVert_{L^p}
    \leqslant 2^{-\frac1p}\lVert G_j\rVert_{H^p}$.
  Since we mainly consider the $L^p$ integrability along $\gamma$ 
  of each functions, $h_{2\pm}$ could be viewed as defined on $\gamma_1$ or 
  $\gamma_3$ by tranlation, and the tranlated functions are still denoted as 
  $h_{2\pm}$ by abusing of notation.
  
  We are going to treat two special cases: $0<s<\sigma$, $t>0$;
  or $s>\sigma$, $t<0$, and the other cases could be proved similarly.
  For the first case, $\Gamma_{s,t}\in\Omega_+$. 
  If $\zeta\in\gamma_{s,t,1}\setminus\{\pm\sigma\}$, then
  \[|CF(P_{s,t}(\zeta))|
    \leqslant \frac{10}{\pi}g_{1+}^*(\zeta)+ h_{2+}(\zeta)
        + \frac{10}{\pi}g_{3+}^*(\zeta)
    = H_1(\zeta).\]
  Although $g_{3+}^*$ is originally defined on $\gamma_3$, we could 
  translate it to a function defined on $\gamma_1$ which is denoted 
  as $g_{3+}^*$ again. We will do the same change accordingly 
  in the following expressions, without further explanation. 
  If $\zeta\in\gamma_{s,t,2}\setminus\{\pm\sigma\}$, then
  \[|CF(P_{s,t}(\zeta))|
    \leqslant h_{1+}(\zeta)+ \frac{10}{\pi}g_{2+}^*(\zeta)+ h_{3+}(\zeta)
    = H_2(\zeta),\]
  and if $\zeta\in\gamma_{s,t,3}\setminus\{\pm\sigma\}$, then
  \[|CF(P_{s,t}(\zeta))|
    \leqslant \frac{10}{\pi}g_{1+}^*(\zeta)+ h_{2+}(\zeta)
        + \frac{10}{\pi}g_{3+}^*(\zeta)
    = H_1(\zeta).\]  
  Define $g(\zeta)= H_j(\zeta)$ when 
  $\zeta\in\gamma_{s,t,j}\setminus\{\pm\sigma\}$ 
  for $j=1$, $2$, $3$, we have
  \[|(CF)_{s,t}(\zeta)|= |CF(P_{s,t}(\zeta))|\leqslant g(\zeta)\quad
    \text{for } \zeta\in\gamma\setminus\{\pm\sigma\}.\]
  
  In the case of $s>\sigma$, $t<0$, let
  \[g(\zeta)= \left\{\!\!
        \begin{array}{ll}
          \frac{10}{\pi}g_{1-}^*(\zeta)
            + \chi_{\mathbb{C}_+}h_{2+}(\zeta)               
            + \chi_{\mathbb{C}_-}h_{2-}(\zeta)
            + \frac{10}{\pi}g_{3+}^*(\zeta) 
            & \text{if $\zeta\in\gamma_{s,t,1}\setminus\{\pm\sigma\}$}, \\
         \chi_{\{\mathrm{Re}\,w<-\sigma\}}h_{1-}(\zeta)
            + \chi_{\{\mathrm{Re}\,w>-\sigma\}}h_{1+}(\zeta)
            + \frac{10}{\pi}g_{2-}^*(\zeta) &                    \\
            \qquad {}+ \chi_{\{\mathrm{Re}\,w<\sigma\}}h_{3+}(\zeta)
            + \chi_{\{\mathrm{Re}\,w>\sigma\}}h_{3-}(\zeta) 
            & \text{if $\zeta\in\gamma_{s,t,2}\setminus\{\pm\sigma\}$}, \\
          \frac{10}{\pi}g_{1+}^*(\zeta)
            + \chi_{\mathbb{C}_-}h_{2-}(\zeta)                
            + \chi_{\mathbb{C}_+}h_{2+}(\zeta)
            + \frac{10}{\pi}g_{3-}^*(\zeta)
            & \text{if $\zeta\in\gamma_{s,t,3}\setminus\{\pm\sigma\}$},
        \end{array}\right.\]
  then we also have, for $\zeta\in\gamma\setminus\{\pm\sigma\}$, 
  $|(CF)_{s,t}(\zeta)|\leqslant g(\zeta)$.
  
  Since the two $g(\zeta)$'s are sum of $L^p$ functions, we know that
  $g(\zeta)\in L^p(\gamma,|\mathrm{d}\zeta|)$.
\end{proof}

The $g_{j\pm}^*$'s and $h_{j\pm}$'s above could even be extended to functions 
defined on $\gamma$ without changing their $L^p$ norm by letting them 
equal to $0$ on parts where they are originally undefined. This point of view
will be very handy in next section. 

The norm of Hardy-Littlewood maximal operator is less than or equal to
$3^{\frac1p}\frac{p}{p-1}$~\cite{Gr08}, then Theorem~\ref{thm-170826-1700}
also leads to the boundedness of Cauchy transform on $\Gamma$. In fact, 
by carefully examning the proof, we know that, 
\[\lVert CF\rVert_{H^p(\Omega_+)}
  \leqslant A_p\Big(\frac{20}{\pi} B_p+2^{1-\frac1p}\Big) 
      \lVert F\rVert_{L^p(\Gamma,|\mathrm{d}\zeta|)},\]
and
\[\lVert CF\rVert_{H^p(\Omega_-)}
  \leqslant A_p\Big(\frac{20}{\pi} B_p+2^{2-\frac1p}\Big) 
      \lVert F\rVert_{L^p(\Gamma,|\mathrm{d}\zeta|)},\]
where $A_p^p= \max\{\frac{p}{p-1}, p^{p-1}\}$, $B_p=3^{\frac1p}\frac{p}{p-1}$. 

\section{Non-tangential Boundary Limit and Cauchy Representation}

In this section, we are going to prove that, if $1<p<\infty$, 
then every function in $H^p(\Omega_\pm)$ has non-tangential boundary limit
a.e.\@ on $\Gamma$, and is the Cauchy integral of its boundary function.
More details are in Theorem~\ref{thm-170827-1120} and 
Theorem~\ref{thm-170827-1410}.

For $\zeta$, $\zeta_0\in\Gamma$, $z\in\mathbb{C}$ and 
$\zeta_0\pm z\neq \zeta$, define
\begin{equation}\label{equ-170826-2220}
  K_z(\zeta,\zeta_0)
  = \frac1{2\pi\mathrm{i}} \Big(\frac1{\zeta-(\zeta_0+z)}
       - \frac1{\zeta-(\zeta_0-z)}\Big)
  = \frac{1}{\pi\mathrm{i}}\cdot \frac{z}{(\zeta-\zeta_0)^2- z^2}.
\end{equation}

\begin{lemma}\label{lem-170826-2230}
  If $\zeta_0+z\in\Omega_+$ and $\zeta_0-z\in\Omega_-$, then
  \[\int_{\Gamma} K_z(\zeta,\zeta_0)\,\mathrm{d}\zeta= 1.\]
\end{lemma}

\begin{proof}
  Choose $t>\max\{0,\mathrm{Im}\,(\zeta_0+z),\mathrm{Im}\,(\zeta_0-z)\}$, 
  and let $E=\Omega_+\cap\{\mathrm{Im}\,w<t\}$, 
  $\Gamma_{E1}=\Gamma\cap\{\mathrm{Im}\,w<t\}$, 
  $\Gamma_{E2}= \{u+\mathrm{i}t\colon |u|\leqslant \sigma\}$
  then $\partial E= \Gamma_{E1}\cup\Gamma_{E2}$, and
  \[\int_{\partial E} K_z(\zeta,\zeta_0)\,\mathrm{d}\zeta
    = \frac1{2\pi\mathrm{i}} \int_{\partial E} 
        \frac{\mathrm{d}\,\zeta}{\zeta-(\zeta_0+z)}
      - \frac1{2\pi\mathrm{i}} \int_{\partial E} 
          \frac{\mathrm{d}\,\zeta}{\zeta-(\zeta_0-z)}
    =1- 0
    =1.\]
  Since 
  \begin{align*}
    \lim_{t\to +\infty} \bigg|\int_{\Gamma_{E2}} 
      \frac{\mathrm{d}\zeta}{(\zeta-\zeta_0)^2-z^2}\bigg|
    &\leqslant \lim_{t\to +\infty} \int_{-\sigma}^\sigma \frac{\mathrm{d}u}
        {|u+\mathrm{i}t-(\zeta_0+z)||u+\mathrm{i}t-(\zeta_0-z)|}        \\
    &\leqslant \lim_{t\to +\infty} \frac{2\sigma}
        {(t-\mathrm{Im}\,(\zeta_0+z))(t-\mathrm{Im}\,(\zeta_0-z))}   \\
    &=0,
  \end{align*}
  we then have, by letting $t\to\infty$, 
  $\int_{\Gamma} K_z(\zeta,\zeta_0)\,\mathrm{d}\zeta= 1$, 
  and the lemma is proved.
\end{proof}

For $\alpha>0$, $\zeta\in\Gamma\setminus\{\pm\sigma\}$, define
\[\Omega_{\alpha\pm}(\zeta)=\left\{\!\!
        \begin{array}{ll}
          \zeta\pm\{x+\mathrm{i}y\colon x>0,|y|< \alpha x\} 
            & \text{if $\zeta\in\Gamma_1$},                          \\
          \zeta\pm\{x+\mathrm{i}y\colon y>0,|x|< \alpha y\} 
            & \text{if $\zeta\in\Gamma_2\setminus\{-\sigma,\sigma\}$},  \\
          \zeta\pm\{x+\mathrm{i}y\colon x<0,|y|< -\alpha x\} 
            & \text{if $\zeta\in\Gamma_3$},                          \\
        \end{array}\right.\]
then $\Omega_{\alpha\pm}(\zeta)$ are cones with vertex $\zeta$.
Notice that both $\Omega_{\alpha+}(\zeta)$ and $\Omega_{\alpha-}(\zeta)$
are not defined for $\zeta=\pm\sigma$.

\begin{lemma}\label{lem-170827-0935}
  If $\alpha>0$ and $\zeta$, $\zeta_0\in\Gamma$ with $\zeta_0\neq \pm\sigma$, 
  then there exists constants $C$, $\delta>0$, depending on 
  $\alpha$, $\zeta_0$, respectively, such that
  \[|K_z(\zeta,\zeta_0)|
    \leqslant \frac{C|z|}{|\zeta-\zeta_0|^2+|z|^2}\]
  for $z+\zeta_0\in \Omega_{\alpha\pm}(\zeta_0)$ and $|z|<\delta$.
\end{lemma}

\begin{proof}
  We could assume $\zeta_0\in\Gamma_1$ and 
  $z+\zeta_0\in\Omega_{\alpha+}(\zeta_0)$, since the other cases could be 
  similarly proved. In view of \eqref{equ-170826-2220}, we need to prove that
  for all $\zeta\in\Gamma$, 
  \[|\zeta-\zeta_0|^2+|z|^2
    \leqslant C_1|(\zeta-\zeta_0)^2-z^2|,\]
  where $z+\zeta_0\in\Omega_{\alpha+}(\zeta_0)$ and $|z|<\delta$ for some
  $C_1$ and $\delta$.
  
  Since $\zeta_0\in\Gamma_1$, let 
  $\delta=\frac12\min\{\mathrm{Im}\,\zeta_0, 2\sigma\}$, then
  $D(\zeta_0,2\delta)\cap\Gamma_1\subset \Gamma_1$, and 
  $D(\zeta_0,2\delta)\cap\Omega_+\subset \Omega_+$. Now choose $|z|<\delta$,
  $\zeta\in \Gamma\setminus D(\zeta_0,2\delta)$, then
  $|\zeta-\zeta_0|\geqslant 2\delta> 2|z|$, and
  \begin{align*}
    |\zeta-\zeta_0|^2+|z|^2
    &\leqslant \frac54 |\zeta-\zeta_0|^2,         \\
    |(\zeta-\zeta_0)^2-z^2|
    &\geqslant |\zeta-\zeta_0|^2-|z|^2
     \geqslant \frac34 |\zeta-\zeta_0|^2,
  \end{align*}
  which implies that
  \[|\zeta-\zeta_0|^2+|z|^2
    \leqslant \frac53 |(\zeta-\zeta_0)^2-z^2|.\]
  
  If $\zeta\in \Gamma\cap D(\zeta_0,2\delta)$, then $\zeta\in\Gamma_1$, 
  and $\arg(\zeta-\zeta_0)= \pm\frac{\pi}2$. 
  Since $|\arg z|<\arctan\alpha$ for 
  $z\in\Omega_{\alpha+}(\zeta_0)-\zeta_0$, we have
  \begin{align*}
    |\zeta-\zeta_0\pm z|
    &= \big||\zeta-\zeta_0|\mathrm{e}^{\mathrm{i}\arg(\zeta-\zeta_0)}
         \pm |z|\mathrm{e}^{\mathrm{i}\arg z}\big|                   \\
    &= \big||\zeta-\zeta_0|\pm |z|\mathrm{e}^{\mathrm{i}\arg z
         - \mathrm{i}\arg(\zeta-\zeta_0)}\big|                   \\
    &\geqslant |z|\cdot|\sin(\arg z-\arg(\zeta-\zeta_0))|           \\
    &\geqslant |z|\cos(\arctan\alpha)                            \\
    &= \frac{|z|}{\sqrt{1+\alpha^2}}.
  \end{align*}
  We also have $|\zeta-\zeta_0\pm z|
    \geqslant |\zeta-\zeta_0|(1+\alpha^2)^{-\frac12}$ by the same method. 
  If, further, $|\zeta-\zeta_0|\leqslant |z|$, then 
  $|\zeta-\zeta_0|^2+|z|^2\leqslant 2|z|^2$, and
  \[|(\zeta-\zeta_0)^2-z^2|
    = |\zeta-\zeta_0+z|\cdot |\zeta-\zeta_0-z|
    \geqslant \frac{|z|^2}{1+\alpha^2};\]
  or if $|\zeta-\zeta_0|> |z|$, then 
  $|\zeta-\zeta_0|^2+|z|^2\leqslant 2|\zeta-\zeta_0|^2$, and
  $|(\zeta-\zeta_0)^2-z^2|\geqslant |\zeta-\zeta_0|^2(1+\alpha^2)^{-1}$.
  In either case, we have
  \[|\zeta-\zeta_0|^2+|z|^2
    \leqslant 2(1+\alpha^2)|(\zeta-\zeta_0)^2-z^2|.\]
  
  Now let $C_1=\max\{2(1+\alpha^2), \frac53\}= 2(1+\alpha^2)$, then for all
  $\zeta\in\Gamma$,
  \[|\zeta-\zeta_0|^2+|z|^2\leqslant C_1|(\zeta-\zeta_0)^2-z^2|,\]
  where $z+\zeta_0\in \Omega_{\alpha+}(\zeta_0)$ and 
  $|z|<\frac12\min\{\mathrm{Im}\,\zeta_0, 2\sigma\}$. This proves the lemma.
\end{proof}

$\Gamma$ could be parametrized in a natural way, that is,
\[\zeta(b)= \left\{\!\!
      \begin{array}{ll}
        -\sigma+(-b-\sigma)\mathrm{i} & \text{if $b< -\sigma$}, \\
        b & \text{if $-\sigma\leqslant b\leqslant \sigma$},    \\
        \sigma+(b-\sigma)\mathrm{i} & \text{if $b> \sigma$},
      \end{array}\right.\]
where $b$ is the signed arc length parameter of $\Gamma$, starting from 
the origin. Then $F(\zeta)$ defined on $\Gamma$ could be considered as 
$F(\zeta(b))$ which is defined on $\mathbb{R}$. Besides,
\[\int_\Gamma F(\zeta)|\mathrm{d}\zeta|
  = \int_{\mathbb{R}} F(\zeta(b))\,\mathrm{d}b.\]
  
\begin{lemma}\label{lem-170827-0820}
  If $\zeta_0=\zeta(b_0)$, $\zeta=\zeta(b)\in\Gamma$ and $\zeta_0$ is fixed,
  then there exists constants $C>0$, depending on $\zeta_0$, such that
  $|\zeta-\zeta_0|\geqslant C|b-b_0|$ for all $\zeta\in\Gamma$.
\end{lemma}

\begin{proof}
  We first deal with the case of $\zeta_0\in\Gamma_1$. If $\zeta\in\Gamma_1$,
  then $|\zeta-\zeta_0|=|b-b_0|$. If $\zeta\in\Gamma_2$, then
  \begin{align*}
    |\zeta-\zeta_0|^2
    &= |b-(-\sigma+(-b_0-\sigma)\mathrm{i})|^2             \\
    &= (b+\sigma)^2+(b_0+\sigma)^2  
     \geqslant \frac12(b-b_0)^2.
  \end{align*}
  The last inequality comes from the elementary inequality 
  $a^2+b^2\geqslant \frac12(a-b)^2$ for $a$, $b\in\mathbb{R}$. It follows that
  $|\zeta-\zeta_0|\geqslant \frac1{\sqrt2}|b-b_0|$. If $\zeta\in\Gamma_3$, 
  we define
  \[g(b)
    = \frac{|\zeta-\zeta_0|^2}{|b-b_0|^2}
    = \frac{4\sigma^2+(b+b_0)^2}{(b-b_0)^2},\]
  where $b>\sigma$ and $b_0< -\sigma$. Since
  \[g'(b)= \frac{-4b_0}{(b-b_0)^3}\big(b+b_0+\frac{2\sigma^2}{b_0}\big),\]
  we know that
  \[\min\{g(b)\colon b>\sigma\}
    = g\big(-b_0-\frac{2\sigma^2}{b_0}\big)
    = \frac{\sigma^2}{b_0^2+\sigma^2}
    < \frac12,\]
  or $|\zeta-\zeta_0|^2\geqslant \frac{\sigma^2}{b_0^2+\sigma^2}|b-b_0|^2$.
  Let $C_1=\min\{1,\frac1{\sqrt2},\frac{\sigma}{\sqrt{b_0^2+\sigma^2}}\}
    = \frac{\sigma}{\sqrt{b_0^2+\sigma^2}}$, then 
  $|\zeta-\zeta_0|\geqslant C_1|b-b_0|$.
  
  Similarly, for all $\zeta\in\Gamma$, if $\zeta_0\in\Gamma_2$, then 
  $|\zeta-\zeta_0|\geqslant \frac1{\sqrt2}|b-b_0|$; if $\zeta_0\in\Gamma_3$, 
  then $|\zeta-\zeta_0|\geqslant C_1|b-b_0|$. Define 
  $C=\min\{C_1, \frac1{\sqrt2}\}
    =\min\{\frac{\sigma}{\sqrt{b_0^2+\sigma^2}}, \frac1{\sqrt2}\}$, then
  $|\zeta-\zeta_0|\geqslant C|b-b_0|$ for all $\zeta\in\Gamma$,
  and the proof is finished.
\end{proof}

\begin{corollary}\label{cor-170827-0930}
  If $1\leqslant p<\infty$, $F(\zeta)\in L^p(\Gamma,|\mathrm{d}\zeta|)$, 
  $\alpha>0$ is fixed, $b_0\neq\pm\sigma$ is the Lebesgue point 
  of $F(\zeta(b))$, then for $z+\zeta_0\in 
    \Omega_{\alpha+}(\zeta_0)\cap\Omega_+$,
  \begin{equation}\label{equ-170827-0940}
    \lim_{z\to 0}\int_\Gamma K_z(\zeta,\zeta_0)F(\zeta)\,\mathrm{d}\zeta
    = F(\zeta_0),
  \end{equation}
  where $\zeta_0=\zeta(b_0)$.
\end{corollary}

\begin{proof}
  Since $\zeta_0+z\in\Omega_{\alpha+}(\zeta_0)\cap\Omega_+$, then 
  $\zeta_0-z\in\Omega_-$ and by Lemma~\ref{lem-170826-2230},
  $\int_{\Gamma} K_z(\zeta,\zeta_0)\,\mathrm{d}\zeta= 1$. 
  Lemma~\ref{lem-170827-0935} shows that there exists $C_1$, $\delta>0$, 
  such that for $\zeta_0+z\in\Omega_{\alpha+}(\zeta_0)$ and $|z|<\delta$,
  we have
  \begin{align*}
    I&= \bigg|\int_\Gamma K_z(\zeta,\zeta_0)F(\zeta)\,\mathrm{d}\zeta
        - F(\zeta_0)\bigg|                           \\
    &= \bigg|\int_\Gamma K_z(\zeta,\zeta_0) (F(\zeta)
        - F(\zeta_0))\,\mathrm{d}\zeta\bigg|               \\
    &\leqslant C_1 \int_\Gamma \frac{|z||F(\zeta)-F(\zeta_0)|}
        {|\zeta-\zeta_0|^2+|z|^2} |\mathrm{d}\zeta|          \\
    &\leqslant C_1 \int_{\mathbb{R}} \frac{|z||F(\zeta(b))-F(\zeta(b_0))|}
            {C_2^2|b-b_0|^2+|z|^2} \mathrm{d}b,
  \end{align*}
  where $C_2>0$ and the last inequality follows from 
  Lemma~\ref{lem-170827-0820}, then
  \begin{align*}
    I
    &\leqslant \frac{\pi C_1}{C_2} \int_{\mathbb{R}} P_{\frac{|z|}{C_2}}(b-b_0)
          |F(\zeta(b))-F(\zeta(b_0))|\,\mathrm{d}b                  \\
    &\leqslant \frac{\pi C_1}{C_2} \int_{\mathbb{R}} P_{\frac{|z|}{C_2}}(b) 
          |F(\zeta(b+b_0))-F(\zeta(b_0))|\,\mathrm{d}b.
  \end{align*}
  Here, $P_x(y)= \frac1\pi\cdot \frac{y}{x^2+y^2}$ is the Poisson kernel on 
  $\mathbb{C}_+$. Since $b_0$ is the Lebesgue point of $F(\zeta(b))$, we have 
  $\lim_{|z|\to 0} I= 0$~\cite{De10}, which is the desired result.
\end{proof}

Obviously, under the condition of Corollary~\ref{cor-170827-0930}, we could 
prove that if $z+\zeta_0\in\Omega_{\alpha-}(\zeta_0)$, then 
\eqref{equ-170827-0940} becomes
\[\lim_{z\to 0}\int_\Gamma K_z(\zeta,\zeta_0)F(\zeta)\,\mathrm{d}\zeta
  = -F(\zeta_0).\] 
We say that function $F(w)$, defined on 
$\Omega_+$ has non-tangential boundary limit $F(\zeta_0)$ at 
$\zeta_0\in\Gamma$, if for all $\alpha>0$,
\[\lim_{w\to\zeta_0} F(w)= F(\zeta_0)\quad
  \text{for } w\in\Omega_{\alpha+}(\zeta_0)\cap\Omega_+.\]
The non-tangential boundary limit of functions on $\Omega_-$ is analogously
defined. Corollary~\ref{cor-170827-0930} tells us that the function
\[G(w)=G(\zeta_0+z)
  = \int_{\Gamma} K_z(\zeta,\zeta_0) F(\zeta)\,\mathrm{d}\zeta\]
has non-tangential boundary limit $F(\zeta_0)$ at $\zeta_0$, although $G(w)$
may be only well-defined in $\Omega_+$ and near $\zeta_0$.

\begin{lemma}\label{lem-170827-1000}
  If $0<p<\infty$, $F(w)\in H^p(\Omega_+)$, $0<s<\sigma$, $t>0$, then
  \[\frac1{2\pi\mathrm{i}} \int_{\Gamma_{s,t}} \frac{F(w)}{w-w_0}\mathrm{d}w
    = \left\{\!\!
        \begin{array}{ll}
          F(w_0) & \text{if $w_0\in D_{s,t}$},\\
          0 & \text{if $w_0\notin \overline{D_{s,t}}$},
        \end{array}\right.\]
\end{lemma}

\begin{proof}
  For fixed $w_0= u_0+\mathrm{i}v_0\notin \Gamma_{s,t}$, 
  let $t_1>\max\{t,v_0\}$, $E= D_{s,t}\cap\{\mathrm{Im}\,w<t_1\}$ with 
  the usual orientation of the boundary, 
  $\Gamma_{E1}= \Gamma_{s,t}\cap\{\mathrm{Im}\,w<t_1\}$,
  $\Gamma_{E2}= \{u+\mathrm{i}t_1\colon |u|\leqslant s\}$, then
  $\partial E= \Gamma_{E1}\cup\Gamma_{E2}$ and $w_0\in D_{s,t}$ implies that
  $w_0\in E$. Since $F(w)$ is analytic,
  \[\frac1{2\pi\mathrm{i}} \int_{\partial E} \frac{F(w)}{w-w_0}\mathrm{d}w
    = \left\{\!\!
        \begin{array}{ll}
          F(w_0) & \text{if $w_0\in D_{s,t}$},\\
          0 & \text{if $w_0\notin \overline{D_{s,t}}$}.
        \end{array}\right.\]
  Define $M(t_1)= \max\{|F(w)|\colon \zeta\in \Gamma_{E2}\}$, then 
  $M(t_1)\to 0$ as $t_1\to +\infty$ by Lemma~\ref{lem-170825-1930}, and
  \begin{align*}
    \bigg|\int_{\Gamma_{E2}} \frac{F(w)}{w-w_0}\mathrm{d}w\bigg|
    &\leqslant M(t_1)\int_{-s}^s \frac{\mathrm{d}u}
        {|u+\mathrm{i}t_1-u_0-\mathrm{i}v_0|}             \\
    &\leqslant M(t_1)\cdot \frac{2s}{t_1-v_0}
     \to 0,
  \end{align*}
  thus
  \[\frac1{2\pi\mathrm{i}} \int_{\Gamma_{s,t}} \frac{F(w)}{w-w_0}\mathrm{d}w
    = \lim_{t_1\to +\infty} \int_{\Gamma_{E1}} \frac{F(w)}{w-w_0}\mathrm{d}w\]
  and the lemma is proved. 
\end{proof}

The $H^p(\Omega_-)$ version of the above lemma is as follows.

\begin{lemma}\label{lem-170827-1010}
  If $1\leqslant p<\infty$, $F(w)\in H^p(\Omega_-)$, $s>\sigma$ and $t<0$, 
  then
  \[\frac1{2\pi\mathrm{i}} \int_{\Gamma_{s,t}} \frac{F(w)}{w-w_0}\mathrm{d}w
    = \left\{\!\!
        \begin{array}{ll}
          0 & \text{if $w_0\in D_{s,t}$},\\
          -F(w_0) & \text{if $w_0\notin \overline{D_{s,t}}$},
        \end{array}\right.\]
\end{lemma}

\begin{proof}
  Fix $w_0= u_0+\mathrm{i}v_0\in \Gamma_{s,t}$. Let 
  $s_1>\max\{s,|u_0|\}$, $t_1<\min\{t,v_0\}$, then
  $D_{s,t}\subset D_{s_1,t_1}$ and $w_0\in D_{s_1,t_1}$.
  Let $t_2>\max\{t,v_0\}$, $E= (D_{s_1,t_1}\setminus\overline{D_{s,t}})
    \cap \{\mathrm{Im} w< t_2\}$ with its boundary be oriented such that 
  $E$ is on the left side of $\partial E$. Define 
  $\Gamma_{E1}= \Gamma_{s_1,t_1}\cap\{\mathrm{Im}\,w<t_2\}$,
  $\Gamma_{E2}= \Gamma_{s,t}\cap\{\mathrm{Im}\,w<t_2\}$,
  $\Gamma_{E3}= \{u+\mathrm{i}t_2\colon s\leqslant |u|\leqslant s_1\}$,
  then $\partial E= \Gamma_{E1}\cup\Gamma_{E2}\cup\Gamma_{E3}$. It is 
  not hard to deduce from Lemma~\ref{lem-170825-2130} that
  \[\frac1{2\pi\mathrm{i}} \bigg(\int_{\Gamma_{s_1,t_1}}
      - \int_{\Gamma_{s,t}}\bigg) \frac{F(w)}{w-w_0}\mathrm{d}w
    = \left\{\!\!
        \begin{array}{ll}
          0 & \text{if $w_0\in D_{s,t}$},\\
          F(w_0) & \text{if $w_0\notin \overline{D_{s,t}}$},
        \end{array}\right.\]  
  
  If $1<p<\infty$, let $\frac1p+\frac1q=1$, then by the proof of 
  Lemma~\ref{lem-170826-0720},
  \begin{align*}
    \bigg|\int_{\Gamma_{s_1,t_1}} 
      \frac{F(w)}{w-w_0}\mathrm{d}w\bigg|
    &\leqslant \bigg(\int_{\Gamma_{s_1,t_1}} 
          |F(w)|^p |\mathrm{d}w|\bigg)^{\frac1p}
        \bigg(\int_{\Gamma_{s_1,t_1}} 
          \frac{|\mathrm{d}w|}{|w-w_0|^q}\bigg)^{\frac1q}           \\
    &\leqslant \lVert F\rVert_{H^p(\Omega_-)} \cdot C^{\frac1p}
        ((s_1+u_0)^{1-p}+(v_0-t_1)^{1-p}+(s_1-u_0)^{1-p})^{\frac1p},
  \end{align*}
  where $C=B(\frac12,\frac{p-1}2)$. If $p=1$, then
  \begin{align*}
    \bigg|\int_{\Gamma_{s_1,t_1}} 
      \frac{F(w)}{w-w_0}\mathrm{d}w\bigg|
    &\leqslant \int_{\Gamma_{s_1,t_1}} |F(w)| |\mathrm{d}w|
        \cdot \sup_{w\in \Gamma_{s_1,t_1}} \frac{1}{|w-w_0|}          \\
    &\leqslant \lVert F\rVert_{H^1(\Omega_-)} \cdot 
        \max\{(s_1-|u_0|)^{-1}, (v_0-t_1)^{-1}\},
  \end{align*}  
  Then the lemma is proved if we let $s_1\to +\infty$ and $t_1\to -\infty$. 
\end{proof}

Now we are in the position of proving the existence of non-tangential buondary
limit of functions in $H^p(\Omega_\pm)$ for $1<p<\infty$.

\begin{theorem}\label{thm-170827-1120}
  If $1<p<\infty$, $F(w)\in H^p(\Omega_+)$, then $F(w)$ has non-tangential 
  boundary limit, which we denote as $F(\zeta)$, a.e.\@ on $\Gamma$, 
  $F(\zeta)\in L^p(\Gamma,|\mathrm{d}\zeta|)$, 
  $\lVert F\rVert_{L^p(\Gamma,|\mathrm{d}\zeta|)}
    \leqslant \lVert F\rVert_{H^p(\Omega_+)}$, and
  \[\frac1{2\pi\mathrm{i}} \int_{\Gamma} 
      \frac{F(\zeta)}{\zeta-w}\mathrm{d}\zeta
    = \left\{\!\!
        \begin{array}{ll}
          F(w) & \text{if $w\in \Omega_+$},\\
          0 & \text{if $w\in \Omega_-$}.
        \end{array}\right.\]
  Besides, $\lVert F_{\sigma-\tau,\tau}
    - \chi_{\Gamma}F\rVert_{L^p(\gamma,|\mathrm{d}\zeta|)}\to 0$
  as $\tau\to 0$, which implies that $\lVert F_{\sigma-\tau,\tau}
    - F\rVert_{L^p(\Gamma,|\mathrm{d}\zeta|)}\to 0$.
  Here, $0<\tau<\sigma$, and $F_{\sigma-\tau,\tau}(\zeta)$ is defined 
  in the same way which is before Lemma~\ref{lem-170826-1650}.
\end{theorem}

\begin{proof}
  Since $0<\tau<\sigma$, $\Gamma_{\sigma-\tau,\tau}\subset \Omega_+$, then
  by definition of $F_{\sigma-\tau,\tau}(\zeta)$,
  \[\int_\gamma |F_{\sigma-\tau,\tau}(\zeta)|^p|\mathrm{d}\zeta|
    = \int_{\Gamma_{\sigma-\tau,\tau}} |F(\zeta_{\sigma-\tau,\tau})|^p 
        |\mathrm{d}\zeta_{\sigma-\tau,\tau}|
    \leqslant \lVert F\rVert_{H^p(\Omega_+)}^p,\]
  where $\gamma= \{\mathrm{Re}\,w= \pm\sigma\}\cup\mathbb{R}$, and it means 
  that $\{F_{\sigma-\tau,\tau}\}$ is bounded in 
  $L^p(\gamma,|\mathrm{d}\zeta|)$. Let $\frac1p+\frac1q=1$, then $1<q<\infty$. 
  Since $L^q(\gamma,|\mathrm{d}\zeta|)$ is seperable Banach space, 
  $\{F_{\sigma-\tau,\tau}\}$ is weak-$\ast$ compact as bounded linear 
  functional on $L^q(\gamma,|\mathrm{d}\zeta|)$, and we could extract a 
  subsequence which weak-$\ast$ converges to a function in 
  $L^p(\gamma,|\mathrm{d}\zeta|)$.
  We denote the subsequence still as $\{F_{\sigma-\tau,\tau}\}$, and the 
  convergence function as $F(\zeta)$ with $\zeta\in\gamma$, then for any
  $G(\zeta)\in L^q(\gamma,|\mathrm{d}\zeta|)$,
  \begin{equation}\label{equ-170827-1150}
    \lim_{\tau\to 0} \int_\gamma F_{\sigma-\tau,\tau}(\zeta)G(\zeta)
        |\mathrm{d}\zeta|
    = \int_\gamma F(\zeta)G(\zeta)|\mathrm{d}\zeta|.
  \end{equation}
  
  Suppose $F(\zeta)\neq 0$ on compact set $E\subset \gamma\setminus\Gamma$ 
  which has positive length measure, we let 
  $G(\zeta)= \chi_E F(\zeta)/|F(\zeta)|$,
  then $G(\zeta)\in L^q(\Gamma,|\mathrm{d}\zeta|)$, 
  and \eqref{equ-170827-1150} becomes
  \[\lim_{\tau\to 0} \int_E F_{\sigma-\tau,\tau}(\zeta)G(\zeta)
        |\mathrm{d}\zeta|
    = \int_E |F(\zeta)||\mathrm{d}\zeta|
    \neq 0.\]
  But if $\tau>0$ is small, we would have $F_{\sigma-\tau,\tau}(\zeta)=0$ on $E$,
  which contradicts with the above limit. Hence $F(\zeta)$ could be replaced 
  with $\chi_\Gamma F(\zeta)$ while in an integral.
  
  For $w_0\notin\Gamma$, there exists $\delta>0$, such that 
  $w_0\notin D_{\sigma+\delta,-\delta}\setminus D_{\sigma-\delta,\delta}$.
  By Lemma~\ref{lem-170827-1000}, if $0<\tau<\delta$, then
  \[\frac1{2\pi\mathrm{i}} \int_{\gamma_{\sigma+\tau,-\tau}} 
        \frac{F_{\sigma-\tau,\tau}(\zeta)}
           {\zeta_{\sigma-\tau,\tau}-w_0}\mathrm{d}\zeta
    = \frac1{2\pi\mathrm{i}} \int_{\Gamma_{\sigma-\tau,\tau}} 
        \frac{F(w)}{w-w_0}\mathrm{d}w
    = \left\{\!\!
        \begin{array}{ll}
          F(w_0) & \text{if $w_0\in \Omega_+$},\\
          0 & \text{if $w_0\in \Omega_-$},
        \end{array}\right.\]
  We let
  \[G(\zeta)= \left\{\!\!
      \begin{array}{ll}
        \frac{\mathrm{i}}{\zeta-w_0} 
          & \text{if $\zeta\in \gamma_{\sigma+\delta,-\delta,1}$},\\
        \frac{1}{\zeta-w_0} 
          & \text{if $\zeta\in \gamma_{\sigma+\delta,-\delta,2}$},\\
        \frac{-\mathrm{i}}{\zeta-w_0} 
          & \text{if $\zeta\in \gamma_{\sigma+\delta,-\delta,3}$},\\
        0 & \text{if $\zeta\in \gamma\setminus\gamma_{\sigma+\delta,-\delta}$},
      \end{array}\right.\]
  then by the proof of Lemma~\ref{lem-170826-0720}, 
  $G(\zeta)\in L^q(\gamma,|\mathrm{d}\zeta|)$, and we rewrite 
  \eqref{equ-170827-1150} as
  \begin{equation}\label{equ-170827-1220}
    \lim_{\substack{0<\tau<\delta,\\ \tau\to0}}
        \int_{\gamma_{\sigma+\tau,-\tau}} 
        \frac{F_{\sigma-\tau,\tau}(\zeta)}{\zeta-w_0}\mathrm{d}\zeta
    = \int_{\Gamma} \frac{F(\zeta)}{\zeta-w_0}\mathrm{d}\zeta
    = \int_{\gamma_{\sigma+\tau,-\tau}} 
        \frac{\chi_\Gamma F(\zeta)}{\zeta-w_0}\mathrm{d}\zeta.
  \end{equation}
  
  Consider
  \begin{align*}
    I&= \int_{\Gamma_{\sigma-\tau,\tau}} 
          \frac{F(w)}{w-w_0}\mathrm{d}w
        - \int_{\Gamma} \frac{F(\zeta)}{\zeta-w_0}\mathrm{d}\zeta    \\
    &= \int_{\gamma_{\sigma+\tau,-\tau}} 
          \bigg(\frac{F_{\sigma-\tau,\tau}(\zeta)}
                  {\zeta_{\sigma-\tau,\tau}-w_0}
          - \frac{\chi_\Gamma F(\zeta)}{\zeta-w_0}\bigg)\mathrm{d}\zeta  \\
    &= \int_{\gamma_{\sigma+\tau,-\tau}} 
          F_{\sigma-\tau,\tau}(\zeta)\bigg(\frac1{\zeta_{\sigma-\tau,\tau}-w_0}
          - \frac1{\zeta-w_0}\bigg)\mathrm{d}\zeta                    \\
    &\qquad{}+ \int_{\gamma_{\sigma+\tau,-\tau}} 
          \frac1{\zeta-w_0}(F_{\sigma-\tau,\tau}(\zeta)
            - \chi_\Gamma F(\zeta))\mathrm{d}\zeta                    \\
    &= I_1+ I_2.
  \end{align*}
  By \eqref{equ-170827-1220}, $I_2\to 0$ as $\tau\to 0$. By definition of
  $F_{\sigma-\tau,\tau}$ and $\zeta_{\sigma-\tau,\tau}$,
  \[|I_1|\leqslant \int_{\gamma_{\sigma-\tau,\tau}} 
           \frac{\tau |F(\zeta_{\sigma-\tau,\tau}(\zeta))|\,|\mathrm{d}\zeta|}
              {|\zeta_{\sigma-\tau,\tau}-w_0||\zeta-w_0|}.\]
  Let $0<\tau<\frac12\delta$, then for all $\zeta\in\Gamma$, we have
  $|\zeta-w_0|\geqslant 2\tau$ and
  \[|\zeta_{\sigma-\tau,\tau}-w_0|
    \geqslant |\zeta-w_0|- |\zeta_{\sigma-\tau,\tau}-\zeta|
    = |\zeta-w_0|- \tau
    \geqslant \frac12|\zeta-w_0|,\]
  thus
  \begin{align*}
    |I_1|
    &\leqslant 2\tau \int_{\gamma_{\sigma-\tau,\tau}} 
        |F_{\sigma-\tau,\tau}(\zeta)|
        \frac{|\mathrm{d}\zeta|}{|\zeta-w_0|^2}                           \\
    &\leqslant 2\tau \bigg(\int_{\gamma_{\sigma-\tau,\tau}} 
          |F_{\sigma-\tau,\tau}(\zeta)|^p |\mathrm{d}\zeta|\bigg)^{\frac1p}
        \bigg(\int_{\gamma_{\sigma-\tau,\tau}} 
          \frac{|\mathrm{d}\zeta|}{|\zeta-w_0|^{2q}}\bigg)^{\frac1q}      \\
    &\leqslant 2\tau \lVert F\rVert_{H^p(\Omega_+)} 
        \lVert G\rVert_{L^{2q}(\gamma,|\mathrm{d}\zeta|)}^2,
  \end{align*}
  which follows that
  \[\lim_{\tau\to 0} |I|\leqslant \lim_{\tau\to 0}(|I_1|+ |I_2|)= 0,\]
  and
  \begin{equation}\label{equ-170827-1350}
    \frac1{2\pi\mathrm{i}} \int_{\Gamma} 
      \frac{F(\zeta)}{\zeta-w_0}\mathrm{d}\zeta
    =\lim_{\tau\to 0} \frac1{2\pi\mathrm{i}} \int_{\Gamma_{\sigma-\tau,\tau}} 
         \frac{F(w)}{w-w_0}\mathrm{d}w
    = \left\{\!\!
        \begin{array}{ll}
          F(w_0) & \text{if $w_0\in \Omega_+$},\\
          0 & \text{if $w_0\in \Omega_-$}.
        \end{array}\right.
  \end{equation}
  
  For $\alpha>0$ fixed, $\zeta_0= \zeta(b_0)\in \Gamma\setminus\{\pm\sigma\}$, 
  where $b_0$ is the Lebesgue point of $F(\zeta(b))$, choose 
  $z\in \Omega_{\alpha+}(\zeta_0)\cap\Omega_+ - \zeta_0$, 
  then $\zeta_0+ z\in\Omega_+$ and $\zeta_0- z\in \Omega_-$. 
  By \eqref{equ-170827-1350},
  \[F(\zeta_0+z)= \frac1{2\pi\mathrm{i}} \int_{\Gamma} 
        \frac{F(\zeta)}{\zeta-(\zeta_0+z)}\mathrm{d}\zeta,\quad
    0= \frac1{2\pi\mathrm{i}} \int_{\Gamma} 
        \frac{F(\zeta)}{\zeta-(\zeta_0-z)}\mathrm{d}\zeta,\]
  then 
  \begin{align*}
    F(\zeta_0+z)
    &= \frac1{2\pi\mathrm{i}} \int_{\Gamma} \Big(\frac1{\zeta-(\zeta_0+z)} 
         - \frac1{\zeta-(\zeta_0-z)}\Big)F(\zeta)\,\mathrm{d}\zeta         \\
    &= \int_{\Gamma} K_z(\zeta,\zeta_0) F(\zeta)\,\mathrm{d}\zeta.
  \end{align*}
  Corollary~\ref{cor-170827-0930} shows that $F(\zeta_0+z)\to F(\zeta_0)$ 
  as $z\to0$, and this implies that $F(w)$ has non-tangential boundary limit 
  $F(\zeta)$ a.e.\@ on $\Gamma$. Thus, 
  $\lVert F\rVert_{L^p(\Gamma,|\mathrm{d}\zeta|)}
    \leqslant \lVert F\rVert_{H^p(\Omega_+)}$ 
  is an easy consequence of Fatou's lemma.
  
  Since $F(w)$ is the Cauchy integral of $F(\zeta)$ on $\Gamma$, then by 
  Theorem~\ref{thm-170826-1700}, $|F_{\sigma-\tau,\tau}(\zeta)|$ is dominated 
  by a function $g(\zeta)\in L^p(\gamma,|\mathrm{d}\zeta|)$, and we deduce 
  from Lebesgue's dominated convergence theorem that,
  \[\lim_{\tau\to 0} \lVert F_{\sigma-\tau,\tau}
      - \chi_{\Gamma}F\rVert_{L^p(\gamma,|\mathrm{d}\zeta|)}= 0\quad
    \text{or }\lim_{\tau\to 0} \lVert F_{\sigma-\tau,\tau}
      - F\rVert_{L^p(\Gamma,|\mathrm{d}\zeta|)}= 0,\]
  and the proof is completed.
\end{proof}

By using the same method as above, we could prove the corresponding theorem on
$H^p(\Omega_-)$

\begin{theorem}\label{thm-170827-1410}
  If $1<p<\infty$, $F(w)\in H^p(\Omega_-)$, then $F(w)$ has non-tangential 
  boundary limit $F(\zeta)\in L^p(\Gamma,|\mathrm{d}\zeta|)$ a.e.\@ 
  on $\Gamma$ with $\lVert F\rVert_{L^p(\Gamma,|\mathrm{d}\zeta|)}
    \leqslant \lVert F\rVert_{H^p(\Omega_-)}$, and
  \[\frac1{2\pi\mathrm{i}} \int_{\Gamma} 
      \frac{F(\zeta)}{\zeta-w}\mathrm{d}\zeta
    = \left\{\!\!
        \begin{array}{ll}
          0 & \text{if $w\in \Omega_+$},\\
          -F(w) & \text{if $w\in \Omega_-$}.
        \end{array}\right.\]
  We also have $\lVert F_{\sigma+\tau,-\tau}
    - \chi_{\Gamma}F\rVert_{L^p(\gamma,|\mathrm{d}\zeta|)}\to 0$
  as $\tau>0$ and $\tau\to 0$.
\end{theorem}

\begin{theorem}[\cite{Vi94}]\label{thm-170828-1040}
  If $1<p\leqslant \infty$, then
  \[H^p(\Omega_+)
    = H^p(\{\mathrm{Re}\,w> -\sigma\})+ H^p(\mathbb{C}_+)
      + H^p(\{\mathrm{Re}\,w< \sigma\}),\]
  in the sense of that in Proposition~\ref{pro-170826-1200}.
\end{theorem}

\begin{proof}
  We only need to prove that functions in $H^p(\Omega_+)$ are sum of 
  functions in the other three $H^p$ spaces. Let $1<p<\infty$, 
  $F(w)\in H^p(\Omega_+)$, then its non-tangential boundary limit 
  $F(\zeta)\in L^p(\Gamma,|\mathrm{d}\zeta|)$, and
  \[F(w)
    = \frac1{2\pi\mathrm{i}} \int_{\Gamma} 
        \frac{F(\zeta)}{\zeta-w}\,\mathrm{d}\zeta
    = \sum_{j=1}^3 \frac1{2\pi\mathrm{i}} \int_{\gamma_j} 
        \frac{\chi_{\Gamma_j}F(\zeta)}{\zeta-w}\,\mathrm{d}\zeta
    = \sum_{j=1}^3 F_j(w).\]
  By Lemma~\ref{lem-170826-1500}, $F_1(w)\in H^p(\{\mathrm{Re}\,w> -\sigma\})$,
  $F_2(w)\in H^p(\mathbb{C}_+)$, $F_3(w)\in H^p(\{\mathrm{Re}\,w< \sigma\})$.
  
  If $p=\infty$, we simply let $F_1(w)$, $F_2(w)$ and $F_3(w)$ be the constant
  $\frac13\lVert F\rVert_{H^p(\Omega_+)}$.
\end{proof}

The following theorem shows that each $L^p(\Gamma,|\mathrm{d}\zeta|)$ function
is the sum of non-tangential boundary limits of two functions in 
$H^p(\Omega_+)$ and $H^p(\Omega_-)$ for $1<p<\infty$, and we usually write 
it as $L^p(\Gamma,|\mathrm{d}\zeta|)= H^p(\Omega_+)+ H^p(\Omega_-)$.

\begin{theorem}
  If $1<p<\infty$, then $F(\zeta)\in L^p(\Gamma,|\mathrm{d}\zeta|)$ if 
  and only if it is the sum of $F_+(\zeta)$ and $F_-(\zeta)$, which are 
  non-tangential boundary limits of $F_+(w)\in H^p(\Omega_+)$ and 
  $F_-(w)\in H^p(\Omega_-)$, respectively.
\end{theorem}

\begin{proof}
  ``$\Leftarrow$'': since $F_+(\zeta)$, 
  $F_-(\zeta)\in L^p(\Gamma,|\mathrm{d}\zeta|)$,
  then $F(\zeta)\in L^p(\Gamma,|\mathrm{d}\zeta|)$.
  
  ``$\Rightarrow$'': For $F(\zeta)\in L^p(\Gamma,|\mathrm{d}\zeta|)$, define
  \[F_+(w_1)= \frac1{2\pi\mathrm{i}} \int_{\Gamma} 
          \frac{F(\zeta)}{\zeta-w_1}\,\mathrm{d}\zeta\quad
    \text{for } w_1\in\Omega_+,\]
  and
  \[F_-(w_2)= -\frac1{2\pi\mathrm{i}} \int_{\Gamma} 
          \frac{F(\zeta)}{\zeta-w_2}\,\mathrm{d}\zeta\quad
    \text{for } w_2\in\Omega_-,\]
  then $F_+(w)\in H^p(\Omega_+)$ and $F_-(w)\in H^p(\Omega_-)$, by 
  Theorem~\ref{thm-170826-1440}. If $b_0\neq\pm\sigma$ is the Lebesgue point 
  of $F(\zeta(b))$, $\alpha>0$ and we choose 
  $z\in \Omega_{\alpha+}(\zeta_0)\cap\Omega_+ - \zeta_0$, 
  then $\zeta_0+ z\in\Omega_+$, $\zeta_0- z\in \Omega_-$ and
  \[F_+(\zeta_0+z)= \frac1{2\pi\mathrm{i}} \int_{\Gamma} 
        \frac{F(\zeta)}{\zeta-(\zeta_0+z)}\mathrm{d}\zeta,\quad
    F_-(\zeta_0-z)= -\frac1{2\pi\mathrm{i}} \int_{\Gamma} 
        \frac{F(\zeta)}{\zeta-(\zeta_0-z)}\mathrm{d}\zeta,\]
  then 
  \[F_+(\zeta_0+z)+ F_-(\zeta_0-z)
    = \int_{\Gamma} K_z(\zeta,\zeta_0) F(\zeta)\,\mathrm{d}\zeta.\]
  By Corollary~\ref{cor-170827-0930},
  \[\lim_{z\to 0} (F_+(\zeta_0+z)+ F_-(\zeta_0-z))
    = \lim_{z\to 0} \int_{\Gamma} K_z(\zeta,\zeta_0) 
         F(\zeta)\,\mathrm{d}\zeta
    = F(\zeta_0),\]
  that is $F(\zeta_0)= F_+(\zeta_0)+ F_-(\zeta_0)$ a.e.\@ on $\Gamma$.
\end{proof}

\begin{lemma}\label{lem-170827-1520}
  If $1\leqslant p\leqslant \infty$, $\frac1p+\frac1q=1$, 
  $F(w)\in H^p(\Omega_+)$, $G(w)\in H^q(\Omega_+)$, $0<s<\sigma$ and $t>0$, 
  then
  \[\int_{\Gamma_{s,t}} F(w)G(w)\,\mathrm{d}w= 0.\]
\end{lemma}

\begin{proof}
  Let $0< s_2< s_1< \sigma$, $0< t_1< t_2< t_3$, then 
  $D_{s_2,t_2}\subset D_{s_1,t_1}$ and 
  \[E= (D_{s_1,t_1}\setminus\overline{D_{s_2,t_2}})
       \cap \{\mathrm{Im}\,w< t_3\}\] is not empty. 
  The boundary of $E$ is 
  \begin{align*}
    \partial E
    &= (\Gamma_{s_1,t_1}\cap\{\mathrm{Im}\,w< t_3\})\cup 
       (\Gamma_{s_2,t_2}\cap\{\mathrm{Im}\,w< t_3\})\cup
       \{u+\mathrm{i}t_3\colon s_2\leqslant |u|\leqslant s_1\}         \\
    &= \Gamma_{E1}\cup \Gamma_{E2}\cup \Gamma_{E3},
  \end{align*}
  with the usual orientation. Since $F(w)G(w)$ is analytic on $\Omega_+$, then
  \[0= \int_{\partial E} F(w)G(w)\,\mathrm{d}w
    = \bigg(\int_{\Gamma_{E1}}- \int_{\Gamma_{E2}}- \int_{\Gamma_{E3}}\bigg) 
         F(w)G(w)\,\mathrm{d}w.\]
  Togother with
  \begin{align*}
    \lim_{t_3\to +\infty}\bigg|\int_{\Gamma_{E3}} F(w)G(w)\,\mathrm{d}w\bigg|
    &\leqslant \lim_{t_3\to +\infty} \int_{s_2\leqslant |u|\leqslant s_1}
         |F(u+\mathrm{i}t_3)G(u+\mathrm{i}t_3)|\,\mathrm{d}u             \\
    &\leqslant 2(s_1-s_2) \lim_{t_3\to +\infty} 
         \max\{|F(w)G(w)|\colon w\in \Gamma_{E3}\}                       \\
    &= 0,
  \end{align*}
  by Lemma~\ref{lem-170825-1930}, and the fact that 
  $F(w)G(w)\in H^1(\Omega_+)$, we have
  \[\int_{\Gamma_{s_1,t_1}}F(w)G(w)\,\mathrm{d}w
    = \int_{\Gamma_{s_2,t_2}}F(w)G(w)\,\mathrm{d}w.\]
  
  Assume $1\leqslant p< \infty$ without loss of generality, if we combine
  \[\bigg|\int_{\Gamma_{\Gamma_{s_2,t_2}}}F(w)G(w)\,\mathrm{d}w\bigg|
    \leqslant \lVert F\rVert_{L^p(\Gamma_{s_2,t_2},|\mathrm{d}w|)}
        \lVert G\rVert_{H^q(\Omega_+)}\]
  and Lemma~\ref{lem-170825-1930}, then, by letting $t_2\to +\infty$,
  \[\int_{\Gamma_{s_1,t_1}}F(w)G(w)\,\mathrm{d}w= 0,\]
  and this proves the lemma.
\end{proof}

\begin{lemma}\label{lem-170827-1530}
  If $1< p< \infty$, $\frac1p+\frac1q=1$, 
  $F(w)\in H^p(\Omega_-)$, $G(w)\in H^q(\Omega_-)$, $s>\sigma$ and $t<0$, 
  then
  \[\int_{\Gamma_{s,t}} F(w)G(w)\,\mathrm{d}w= 0.\]
\end{lemma}

\begin{proof}
  Let $\sigma< s_1< s_2$, $0> t_1> t_2$, by arguing as in 
  Lemma~\ref{lem-170827-1520}, we have
  \[\int_{\Gamma_{s_1,t_1}}F(w)G(w)\,\mathrm{d}w
    = \int_{\Gamma_{s_2,t_2}}F(w)G(w)\,\mathrm{d}w.\]
  and, by supposing $1< p< \infty$,
  \begin{align*}
    \bigg|\int_{\Gamma_{s_2,t_2}}F(w)G(w)\,\mathrm{d}w\bigg|
    &\leqslant \bigg(\int_{\Gamma_{s_2,t_2}} |F(w)|^p |\mathrm{d}w|\bigg)^{\frac1p}
        \lVert G\rVert_{H^q(\Omega_+)}                                \\
    &= \bigg(\int_{\gamma_{s_2,t_2}} |F_{s_2,t_2}(\zeta)|^p 
          |\mathrm{d}\zeta| \bigg)^{\frac1p} \lVert G\rVert_{H^q(\Omega_+)}.
  \end{align*}
  Since Theorem~\ref{thm-170827-1410} and Theorem~\ref{thm-170826-1700} 
  imply that $|F_{s_2,t_2}(\zeta)|$ is dominated by a function in 
  $L^p(\gamma,|\mathrm{d}\zeta|)$, and Lemma~\ref{lem-170825-2140} 
  shows that $|F_{s_2,t_2}(\zeta)|\to 0$ as $|s_2|$,
  $|t_2|\to +\infty$, we have, by Lebesgue's dominated convergence theorem
  \[\lim_{|s_2|,\,|t_2|\to 0} \int_{\gamma_{s_2,t_2}} 
      |F_{s_2,t_2}(\zeta)|^p |\mathrm{d}\zeta|= 0,\]
  then 
  \[\int_{\Gamma_{s_1,t_1}}F(w)G(w)\,\mathrm{d}w= 0,\]
  and the proof is finished.
\end{proof}

\begin{proposition}\label{pro-170827-1630}
  If $1<p<\infty$, $\frac1p+\frac1q=1$, $F(w)\in H^p(\Omega_+)$, 
  $G(w)\in H^q(\Omega_+)$ and $F(\zeta)$, $G(\zeta)$ are the corresponding 
  non-tangential boundary limit on $\Gamma$, then
  \[\int_\Gamma F(\zeta)G(\zeta)\,\mathrm{d}\zeta= 0.\] 
\end{proposition}

\begin{proof}
  For $0<\tau<\sigma$, by Lemma~\ref{lem-170827-1520} and definition of 
  $F_{\sigma-\tau,\tau}(\zeta)$, we have
  \[\int_{\gamma_{\sigma-\tau,\tau}} F_{\sigma-\tau,\tau}(\zeta)
        G_{\sigma-\tau,\tau}(\zeta)\,\mathrm{d}\zeta
    = \int_{\Gamma_{\sigma-\tau,\tau}}F(w)G(w)\,\mathrm{d}w= 0.\]
  Since $F(\zeta)$ could be replaced by $\chi_\Gamma F(\zeta)$ while 
  in integrand, we then have
  \begin{align*}
    &\bigg|\int_\Gamma F(\zeta)G(\zeta)\,\mathrm{d}\zeta
      - \int_{\gamma_{\sigma-\tau,\tau}} F_{\sigma-\tau,\tau}(\zeta)
            G_{\sigma-\tau,\tau}(\zeta)\,\mathrm{d}\zeta \bigg|            \\
    \leqslant{}& \bigg|\int_\gamma \chi_\Gamma F(\zeta)(\chi_\Gamma G(\zeta)
         - G_{\sigma-\tau,\tau}(\zeta))\,\mathrm{d}\zeta\bigg|
      + \bigg|\int_\gamma (F_{\sigma-\tau,\tau}(\zeta)-\chi_\Gamma F(\zeta)) 
           G_{\sigma-\tau,\tau}(\zeta))\,\mathrm{d}\zeta\bigg|             \\
    \leqslant{}& \lVert F\rVert_{H^p(\Omega_+)} \lVert \chi_\Gamma G
        - G_{\sigma-\tau,\tau}\rVert_{L^q(\gamma,|\mathrm{d}\zeta|)}
      + \lVert F_{\sigma-\tau,\tau}- \chi_\Gamma F
           \rVert_{L^q(\gamma,|\mathrm{d}\zeta|)}
        \lVert G\rVert_{H^q(\Omega_+)},
  \end{align*}
  which tends to $0$ as $\tau\to 0$ by Theorem~\ref{thm-170827-1120}.
\end{proof}

\begin{proposition}\label{pro-170827-2304}
  If $1<p<\infty$, $\frac1p+\frac1q=1$, $F(w)\in H^p(\Omega_-)$, 
  $G(w)\in H^q(\Omega_-)$, and $F(\zeta)$, $G(\zeta)$ are their 
  non-tangential boundary limits, then
  \[\int_\Gamma F(\zeta)G(\zeta)\,\mathrm{d}\zeta= 0.\] 
\end{proposition}

The proof is the same as above. We now give a characterization of 
$L^p(\Gamma,|\mathrm{d}\zeta|)$ functions be the non-tangential boundary 
limit of $H^p(\Omega_\pm)$ functions, where $1<p<\infty$.

\begin{theorem}
  If $1<p<\infty$, $F(\zeta)\in L^p(\Gamma,|\mathrm{d}\zeta|)$, 
  then $F(\zeta)$ is the non-tangential boundary limit of a function in 
  $H^p(\Omega_+)$ if and only if
  \[\frac1{2\pi\mathrm{i}} \int_\Gamma \frac{F(\zeta)}{\zeta-\alpha}
      \mathrm{d}\zeta= 0 \quad
    \text{for all } \alpha\in\Omega_-.\]
\end{theorem}

\begin{proof}
  ``$\Rightarrow$'': let $\frac1p+\frac1q=1$, then $1<q<\infty$, 
  $G(w)=\frac1{w-\alpha}\in H^q(\Omega_+)$ for $\alpha\in\Omega_-$ by 
  Corollary~\ref{cor-170826-0740}, 
  and has non-tangential boundary limit $G(\zeta)=\frac1{\zeta-\alpha}$ 
  a.e.\@ on $\Gamma$. By Proposition~\ref{pro-170827-1630},
  \[\int_\Gamma F(\zeta)G(\zeta)\,\mathrm{d}\zeta= 0\quad
    \text{or } \int_\Gamma \frac{F(\zeta)}{\zeta-\alpha}\mathrm{d}\zeta= 0.\]
  
  ``$\Leftarrow$'': define
  \[G(w)= \frac1{2\pi\mathrm{i}} \int_\Gamma \frac{F(\zeta)}{\zeta-w} \mathrm{d}\zeta
    \quad\text{for } w\in\Omega_+,\]
  then $G(w)\in H^p(\Omega_+)$ by Theorem~\ref{thm-170826-1440}, thus has non-tangential
  boundary limit $G(\zeta)$ a.e.\@ on $\Gamma$. 
  Fix $\zeta_0=\zeta(b_0)\in\Gamma\setminus\{\pm\sigma\}$ where $b_0$ is the Lebesgue
  point of both $F(\zeta(b))$ and $G(\zeta(b))$, let $\alpha>0$, 
  $z+\zeta_0\in \Omega_{\alpha+}\cap\Omega_+$, then $\zeta_0-z\in \Omega_-$ 
  and
  \[\frac1{2\pi\mathrm{i}} \int_\Gamma \frac{F(\zeta)}{\zeta-(\zeta_0-z)} 
      \mathrm{d}\zeta=0,\]
  which follows that
  \begin{align*}
    G(\zeta_0+z)
    &= \frac1{2\pi\mathrm{i}} \int_{\Gamma} 
        \frac{F(\zeta)}{\zeta-(\zeta_0+z)}\mathrm{d}\zeta
      - \frac1{2\pi\mathrm{i}} \int_{\Gamma} 
        \frac{F(\zeta)}{\zeta-(\zeta_0-z)}\mathrm{d}\zeta                 \\
    &= \int_{\Gamma} K_z(\zeta,\zeta_0) F(\zeta)\,\mathrm{d}\zeta.
  \end{align*}
  By Corollary~\ref{cor-170827-0930},
  \[\lim_{z\to 0} G(\zeta_0+z)= F(\zeta_0)\quad
    \text{or } G(\zeta_0)= F(\zeta_0),\]
  that is, $F(\zeta)$ is the non-tangential boundary limit function of 
  $G(w)\in H^p(\Omega_+)$ a.e.\@ on $\Gamma$.
\end{proof}

We have the following characterization of the non-tangential boundary limit 
of $H^p(\Omega_\pm)$ functions with $1<p<\infty$.

\begin{theorem}
  If $1<p<\infty$, $F(\zeta)\in L^p(\Gamma,|\mathrm{d}\zeta|)$, then 
  $F(\zeta)$ is the non-tangential boundary limit of a function in 
  $H^p(\Omega_-)$ if and only if
  \[\frac1{2\pi\mathrm{i}} \int_\Gamma \frac{F(\zeta)}{\zeta-\alpha}
      \mathrm{d}\zeta= 0 \quad
    \text{for all } \alpha\in\Omega_+.\]
\end{theorem}

\section{Isomorphism of $H^p(\mathbb{C}_\pm)$ and $H^p(\Omega_\pm)$}

We will prove that if $0<p<\infty$, then $H^p(\mathbb{C}_+)$ and 
$H^p(\mathbb{C}_-)$ are isomorphic to $H^p(\Omega_+)$ and $H^p(\Omega_-)$, 
respectively, under proper defined transforms. Then $H^p(\Omega_+)$ is 
isomorphic to $H^p(\Omega_-)$, since $H^p(\mathbb{C}_+)$ and 
$H^p(\mathbb{C}_-)$ are isometric to each other. Most of
our results here are parallel to those in~\cite{DL171} and \cite{DL172}, 
often with exactly the same proving method, although there $\Omega_+$ 
is the domain over a Lipschitz curve.

Since $\Omega_+$ and $\Omega_-$ are simply connected domains, then by 
Riemann mapping theorem, there exists holomorphic representations $\Phi_+(z)$ 
from $\mathbb{C_+}$ onto $\Omega_+$, and $\Phi_-(z)$ from $\mathbb{C_-}$ 
onto $\Omega_-$. We denote the inverse of $\Phi_+(z)$ as $\Psi_+(w)$ and 
that of $\Phi_-(z)$ as $\Psi_-(w)$. All of them extend to the boundaries, 
and we let the extensions on the boundaries be $\Phi_\pm(x)$ for 
$x\in\mathbb{R}$ and 
$\Psi_\pm(\zeta)$ for $\zeta\in\Gamma$, then $\Phi_\pm'(z)\to\Phi_\pm'(x)$,  
$\Psi_\pm'(w)\to\Psi_\pm'(\zeta)$ non-tangentially a.e.\@, where the latters 
are derivatives along the boundaries. If $\Phi_+(z)=w$ for 
$\overline{\mathbb{C}_+}$, then $\Phi_+'(z)\Psi_+'(w)=1$ a.e.\@. The same is true 
for $\Phi_-(z)$ and $\Psi_-(w)$.

Without loss of generality, we suppose $\Phi_\pm(-1)= -\sigma$ and 
$\Phi_\pm(1)= \sigma$, then by Schwarz-Christoffel formula~\cite{SZ71}, 
\[\Phi_+(z)= \frac{2\sigma}{\pi} \int_0^z 
    \frac{\mathrm{d}\xi_1}{\sqrt{1-\xi_1^2}}
  \quad \text{and}\quad
  \Phi_-(z)= \frac{4\sigma}{\pi} \int_0^z 
    \sqrt{1-\xi_2^2}\mathrm{d}\xi_2.\]
Here, we choose the branch of $\sqrt{1-\xi_1^2}$ which makes it analytic 
on $\mathbb{C}_+$ and positive when $\xi_1\in (-1,1)\subset\mathbb{R}$, 
and that of $\sqrt{1-\xi_2^2}$ which makes it analytic on $\mathbb{C}_-$ 
and positive when $\xi_2\in (-1,1)\subset\mathbb{R}$. More specifically, 
for $\xi_1\in\mathbb{C}_+$, $\arg(1-\xi_1)\in (-\pi,0)$ and 
$\arg(1+\xi_1)\in (0,\pi)$, while for $\xi_2\in\mathbb{C}_-$, 
$\arg(1-\xi_2)\in (0,\pi)$ and $\arg(1+\xi_1)\in (-\pi,0)$.
Actually, one could verify that $\Phi_+(z)= \frac{2\sigma}{\pi} \arcsin z$ 
with principle value in $\Omega_+$, 
and $\Psi_+(w)= \sin(\frac{\pi}{2\sigma}w)$.

For $0<p<\infty$, define transform $T_+$ from $H^p(\Omega_+)$ to analytic 
functions on $\mathbb{C}_+$ as
\begin{equation}\label{equ-170827-1810}
  T_+F(z)= F(\Phi_+(z))(\Phi_+'(z))^{\frac1p} \quad
  \text{for } F(w)\in H^p(\Omega_+),
\end{equation}
and transform $T_-$ from $H^p(\Omega_-)$ to analytic functions on 
$\mathbb{C}_-$ as
\begin{equation}\label{equ-170827-1820}
  T_-F(z)= F(\Phi_-(z))(\Phi_-'(z))^{\frac1p} \quad
  \text{for } F(w)\in H^p(\Omega_-),
\end{equation}
then both $T_+$ and $T_-$ are one-to-one and linear. If $p=\infty$, 
then $T_\pm$ become $F(\Phi_\pm)$, which implies $H^\infty(\mathbb{C}_\pm)$ 
are isometric to $H^\infty(\Omega_\pm)$.

Let $D$ be an arbitrary simply connected domain with at least 
two boundary points. A function~$f$ analytic on $D$ is said to be of 
class~$E^p(D)$~\cite{Du70} if there exists a sequence of rectifiable 
Jordan curves~$C_1$, $C_2$, $\ldots$ in $D$, which eventually 
surround each compact subdomain of $D$, such that
\[\sup_{n\geqslant 1}\int_{C_n} |f(z)|^p |\mathrm{d}z|< \infty.\]

\begin{lemma}\label{lem-170827-1800}
  If $0<p<\infty$, $T_+$ is defined on $E^p(\Omega_+)$ as in 
  \eqref{equ-170827-1810}, and $T_-$ is defined on $E^p(\Omega_-)$ as in 
  \eqref{equ-170827-1820}, then $T_+(E^p(\Omega_+))\subset H^p(\mathbb{C}_+)$ 
  and $T_-(E^p(\Omega_-))\subset H^p(\mathbb{C}_-)$.
\end{lemma}

The proof of the above lemma is exactly the same as in~\cite{DL172}, so we 
omit it here.

\begin{proposition}\label{pro-170827-2150}
  If $0<p<\infty$, then for $T_+$ defined on $H^p(\Omega_+)$ by 
  \eqref{equ-170827-1810}, $T_+(H^p(\Omega_+))\subset H^p(\mathbb{C}_+)$. 
  In addition, $\lVert T_+\rVert\leqslant 1$ for $1<p<\infty$.
\end{proposition}

\begin{proof}
  We only need to verify that $H^p(\Omega_+)\subset E^p(\Omega_+)$. 
  For $n\in\mathbb{N}$, let $E_n= D_{\frac{n\sigma}{n+1},\frac1{n+1}}
    \cap \{\mathrm{Im}\,w< n\}$ and
  \[C_n=\partial E_n
    = (\Gamma_{\frac{n\sigma}{n+1},\frac1{n+1}}\cap \{\mathrm{Im}\,w< n\})
      \cup \{u+\mathrm{i}n\colon |u|\leqslant n\sigma/(n+1)\}
    = \Gamma_{n1}\cup \Gamma_{n2},\]
  then $E_n\neq\emptyset$ and $E_n\to\Omega_+$. If $F(w)\in H^p(\Omega_+)$, then
  \begin{align*}
    \int_{C_n} |F(w)|^p |\mathrm{d}w|
    &= \bigg(\int_{\Gamma_{n1}}+ \int_{\Gamma_{n2}}\bigg)|F(w)|^p 
         |\mathrm{d}w|                                             \\
    &\leqslant \bigg(\int_{\Gamma_{\frac{n\sigma}{n+1},\frac1{n+1}}}
         + \int_{\Gamma_{\frac{n\sigma}{n+1},n}}\bigg)|F(w)|^p |\mathrm{d}w| \\
    &\leqslant 2\lVert F\rVert_{H^p(\Omega_+)}^p,
  \end{align*}
  which follows that $F(w)\in E^p(\Omega_+)$ and $T_+F(z)\in H^p(\mathbb{C}_+)$.
  
  If $1<p<\infty$, then $F(w)$ has non-tangential boundary limit $F(\zeta)$ 
  a.e.\@ on $\Gamma$, and by Fatou's lemma,
  \[\lVert T_+F\rVert_{H^p(\mathbb{C}_+)}^p
    = \int_{\mathbb{R}} |T_+F(x)|^p\mathrm{d}x
    = \int_\Gamma |F(\zeta)|^p |\mathrm{d}\zeta|
    \leqslant \lVert F\rVert_{H^p(\Omega_+)}^p,\]
  which shows that $\lVert T_+\rVert\leqslant 1$.
\end{proof}

\begin{proposition}\label{pro-170828-0940}
  If $0<p<\infty$, then for $T_-$ defined on $H^p(\Omega_-)$ by 
  \eqref{equ-170827-1820}, $T_-(H^p(\Omega_-))\subset H^p(\mathbb{C}_-)$. 
  Besides, $\lVert T_-\rVert\leqslant 1$ for $1<p<\infty$.
\end{proposition}

\begin{proof}
  We should also verify that $H^p(\Omega_-)\subset E^p(\Omega_-)$. 
  Fix $n\in\mathbb{N}$, $F(w)\in H^p(\Omega_-)$, let 
  \[E_n= D_{(n+2)\sigma,-n}\setminus
     \overline{D_{\frac{n+1}n\sigma,-\frac1{n+1}}},\]
  then by Lemma~\ref{lem-170825-2130},
  \begin{align*}
    &\lim_{t\to +\infty} \bigg(\int_{-(n+2)\sigma}^{-\frac{n+1}n\sigma}
      + \int_{\frac{n+1}n\sigma}^{(n+2)\sigma}\bigg) 
         |F(u+\mathrm{i}t)|^p\mathrm{d}u                           \\
    \leqslant{}& \lim_{t\to +\infty} 2\sigma\Big(n+1-\frac1n\Big) 
        \max\{|F(u+\mathrm{i}t)| \colon  (n+1)\sigma/n\leqslant 
           |u|\leqslant (n+2)\sigma\}                               \\
    ={}& 0,
  \end{align*}
  and we could choose $t_n>n$, such that
  \[\bigg(\int_{-(n+2)\sigma}^{-\frac{n+1}n\sigma}
      + \int_{\frac{n+1}n\sigma}^{(n+2)\sigma}\bigg) 
          |F(u+\mathrm{i}t_n)|^p\mathrm{d}u
    <1.\]
  Now define $C_n$ as the boundary of $E_n\cap\{\mathrm{Im}\,w< t_n\}$, then
  \begin{align*}
    \int_{C_n} |F(w)|^p |\mathrm{d}w|
    &\leqslant \bigg(\int_{\gamma_{\frac{n+1}n\sigma,-\frac1{n+1}}}
        + \int_{\gamma_{(n+2)\sigma,-n}}\bigg)|F(w)|^p |\mathrm{d}w|+ 1  \\
    &\leqslant 2\lVert F\rVert_{H^p(\Omega_-)}^p+ 1.
  \end{align*}
  Since $E_n\cap\{\mathrm{Im}\,w< t_n\}\to\Omega_-$, we have 
  $F(w)\in E^p(\Omega_-)$. The boundedness of $\lVert T_-\rVert$ when 
  $1<p<\infty$ also comes from Fatou's lemma.
\end{proof}

Remeber that $\Phi_+'(z)= \frac{2\sigma}{\pi\sqrt{1-z^2}}$ and 
$\Phi_-'(z)= \frac{4\sigma}{\pi} \sqrt{1-z^2}$, both with properly 
chosen branch.

\begin{lemma}\label{lem-170827-1920}
  If $y>0$, then $\mathrm{Re}\,\Phi_+'(x+\mathrm{i}y)>0$, 
  $x\mathrm{Im}\,\Phi_+'(x+\mathrm{i}y)>0$ for $x\neq 0$, and 
  $\mathrm{Im}\,\Phi_+'(\mathrm{i}y)=0$.  
  Also, $\mathrm{Re}\,\Phi_-'(x-\mathrm{i}y)>0$, 
  $x\mathrm{Im}\,\Phi_-'(x-\mathrm{i}y)>0$ for $x\neq 0$, and 
  $\mathrm{Im}\,\Phi_-'(-\mathrm{i}y)=0$.  
\end{lemma}

\begin{proof}
  We only prove the $\Phi_+'$ case and the $\Phi_-'$ case could be 
  similarly proved. Since 
  $\Phi_+'(z)= \frac{2\sigma}{\pi}(1-z)^{-\frac12}(1+z)^{-\frac12}$ with
  $\arg(1-z)\in (-\pi,0)$ and $\arg(1+z)\in (0,\pi)$ for $z\in\mathbb{C}_+$, 
  we have 
  \[\arg\Phi_+'(z)= -\frac12(\arg(1-z)+\arg(1+z))
      \in\Big(-\frac\pi2,\frac\pi2\Big),\]
  then $\mathrm{Re}\,\Phi_+'(z)>0$. Let $z= x+\mathrm{i}y\in \mathbb{C}_+$, 
  then $y>0$, $1-z=1-x-\mathrm{i}y$ and $1+z=1+x+\mathrm{i}y$. 
  
  If $x<-1$, then $\arg(1-z)\in (-\frac\pi2,0)$, 
  $\arg(1+z)\in (\frac\pi2,\pi)$, and $\arg\Phi_+'(z)\in (-\frac\pi2,0)$;
  
  If $x=-1$, then $\arg(1-z)\in (-\frac\pi2,0)$, 
    $\arg(1+z)=\frac\pi2$, and $\arg\Phi_+'(z)\in (-\frac\pi4,0)$;
  
  If $-1<x<0$, then 
  \[\arg(1-z)= \arctan\frac{-y}{1-x},\quad \arg(1+z)= \arctan\frac{y}{1+x},\]
  and 
  \[\arg\Phi_+'(z)
    = \frac12\Big(\arctan\frac{y}{1-x}- \arctan\frac{y}{1+x}\Big)
    \in (-\pi/4,0).\]
  In each case, $\mathrm{Im}\,\Phi_+'(z)<0$.
  
  If $x=0$, then $\arg(1-z)=\arctan(-y)$, $\arg(1+z)= \arctan y$, 
  and $\arg\Phi_+'(\mathrm{i}y)= 0$ which means that 
  $\mathrm{Im}\,\Phi'_+(\mathrm{i}y)= 0$. If $x>0$, we analyse the 
  three cases of $0<x<1$, $x=1$ and $x>1$, and would have 
  $\mathrm{Im}\,\Phi'_+(x+\mathrm{i}y)>0$. Then we have proved the lemma.
\end{proof}

\begin{lemma}\label{lem-170827-2105}
  Suppose that $1<q<\infty$, $\alpha\in\mathbb{C}$, $\varepsilon>0$, 
  and let $E(\alpha,\varepsilon)=\{z\in\mathbb{C}_+\colon
      |\Phi_+(z)-\alpha|\geqslant \varepsilon\}$.
  Let $E_y=\{t\in\mathbb{R}\colon t+\mathrm{i}y\in E(\alpha,\varepsilon)\}$ 
  for $y>0$, then
  \[I= \int_{E_y} \frac{\lvert\Phi_+'(t+\mathrm{i}y)\rvert\,\mathrm{d}t}
         {\lvert\Phi_+(t+\mathrm{i}y)-\alpha\rvert^q}
     \leqslant \frac{3\cdot2^{q+1}}{(q-1)\varepsilon^{q-1}}.\]

  As a consequence, if $\alpha\in \Omega_-$, and we define
  \[g(z)=\frac{\big(\Phi_+'(z)\big)^{\frac1q}}{\Phi_+(z)-\alpha},\quad
    \text{for } z\in\mathbb{C}_+,\]
  then $g(z)\in H^q(\mathbb{C}_+)$.
\end{lemma}

\begin{proof}
  Since $\lvert\Phi_+(z)-\alpha\rvert\geqslant \varepsilon$ 
  for $z\in E(\alpha,\varepsilon)$, then 
  \begin{align*}
    \lvert \Phi_+(z)-\alpha\rvert
    &\geqslant \frac12(\lvert \mathrm{Re}\,\Phi(z)
          - \mathrm{Re}\,\alpha\rvert+ \varepsilon),        \\
    \lvert \Phi_+(z)-\alpha\rvert
    &\geqslant \frac12(\lvert \mathrm{Im}\,\Phi(z)
          - \mathrm{Im}\,\alpha\rvert+ \varepsilon).
  \end{align*}
  For fixed $y>0$, define $h(t)= \mathrm{Re}\,\Phi_+(t+\mathrm{i}y)$, then
  \[h'(t)= \frac{\mathrm{d}}{\mathrm{d}t} \mathrm{Re}\,\Phi_+(t+\mathrm{i}y)
    = \mathrm{Re}\,\Phi_+'(t+\mathrm{i}y)
    >0,\]
  which shows $\mathrm{Re}\,\Phi_+(t+\mathrm{i}y)$ is an increasing function 
  of $t$. Similarly, $\mathrm{Im}\,\Phi_+(t+\mathrm{i}y)$ as a function 
  of $t$ is decreasing if $t\leqslant 0$ while increasing if $t>0$, then
  \begin{align*}
    I_1
    &= \int_{E_y} \frac{|\mathrm{Re}\,\Phi_+'(t+\mathrm{i}y)|\,\mathrm{d}t}
             {\lvert\Phi_+(t+\mathrm{i}y)-\alpha\rvert^q}       \\      
    &\leqslant \int_{E_y} \frac{\mathrm{d}\,\mathrm{Re}\,\Phi_+(t+\mathrm{i}y)}
        {2^{-q}(\lvert \mathrm{Re}\,\Phi_+(t+\mathrm{i}y)
          -\mathrm{Re}\,\alpha\rvert+ \varepsilon)^q}         \\
    &\leqslant \int_{\mathbb{R}}\frac{2^q\,\mathrm{d}t}{(|t|+\varepsilon)^q}   
     = \frac{2^{q+1}}{(q-1)\varepsilon^{q-1}},
  \end{align*}
  and
  \begin{align*}
    I_2
    &= \int_{E_y} \frac{\lvert\mathrm{Im}\,\Phi_+'(t+\mathrm{i}y)\rvert
           \,\mathrm{d}t} {\lvert\Phi_+(t+\mathrm{i}y)-\alpha\rvert^q}   \\
    &\leqslant \int_{E_y\cap\mathbb{R}_-} 
        \frac{-2^q\mathrm{d}\,\mathrm{Im}\,\Phi_+(t+\mathrm{i}y)}
        {(\lvert \mathrm{Im}\,\Phi_+(t+\mathrm{i}y)
          -\mathrm{Im}\,\alpha\rvert+ \varepsilon)^q}                  \\
    &\qquad {}+ \int_{E_y\cap\mathbb{R}_+} 
          \frac{2^q\mathrm{d}\,\mathrm{Im}\,\Phi_+(t+\mathrm{i}y)}
            {(\lvert \mathrm{Im}\,\Phi_+(t+\mathrm{i}y)
              -\mathrm{Im}\,\alpha\rvert+ \varepsilon)^q}              \\
    &\leqslant 2\int_{\mathbb{R}}\frac{2^q\,\mathrm{d}t}{(|t|+\varepsilon)^q}   
     = \frac{2^{q+2}}{(q-1)\varepsilon^{q-1}},
  \end{align*}
  It follows that
  \[I\leqslant I_1+I_2
    \leqslant \frac{3\cdot 2^{q+1}}{(q-1)\varepsilon^{q-1}}.\]
  
  If $\alpha\in \Omega_-$, then there exists $\varepsilon>0$, 
  such that $\lvert\Phi_+(z)-\alpha\rvert\geqslant \varepsilon$ 
  for all $z\in\mathbb{C}_+$. Hence $E_y=\mathbb{R}$, and for $y>0$,
  \[\int_{\mathbb{R}} \lvert g(t+\mathrm{i}y)\rvert^q \,\mathrm{d}t
    = \int_{\mathbb{R}} \frac{\lvert\Phi_+'(t+\mathrm{i}y)\rvert\,\mathrm{d}t}
           {\lvert\Phi_+(t+\mathrm{i}y)-\alpha\rvert^q}
    \leqslant \frac{3\cdot 2^{q+1}}{(q-1)\varepsilon^{q-1}},\]
  which implies that $g(z)\in H^q(\mathbb{C}_+)$, as the boundary above 
  is independent of $y$.
\end{proof}

Obviously, Lemma~\ref{lem-170827-2105} has a $\Phi_-(z)$ version.

\begin{lemma}
  Suppose that $1<q<\infty$, $\alpha\in\mathbb{C}$, $\varepsilon>0$, 
  and let $E(\alpha,\varepsilon)=\{z\in\mathbb{C}_-\colon
      |\Phi_-(z)-\alpha|\geqslant \varepsilon\}$.
  Let $E_y=\{t\in\mathbb{R}\colon t+\mathrm{i}y\in E(\alpha,\varepsilon)\}$ 
  for $y<0$, then
  \[I= \int_{E_y} \frac{\lvert\Phi_-'(t+\mathrm{i}y)\rvert\,\mathrm{d}t}
         {\lvert\Phi_-(t+\mathrm{i}y)-\alpha\rvert^q}
     \leqslant \frac{3\cdot2^{q+1}}{(q-1)\varepsilon^{q-1}}.\]

  Consequently, if $\alpha\in \Omega_+$, and we define
  \[g(z)=\frac{\big(\Phi_-'(z)\big)^{\frac1q}}{\Phi_-(z)-\alpha},\quad
    \text{for } z\in\mathbb{C}_-,\]
  then $g(z)\in H^q(\mathbb{C}_-)$.
\end{lemma}

\begin{proposition}\label{pro-170827-2222}
  If $1<p<\infty$, then for $T_+$ defined by \eqref{equ-170827-1810}, we have
  $H^p(\mathbb{C}_+)\subset T_+(H^p(\Omega_+))$, 
  or $T_+^{-1}(H^p(\mathbb{C}_+))\subset H^p(\Omega_+)$. Here
  $T_+^{-1}f(w)= f(\Psi_+(w))(\Psi_+'(w))^{\frac1p}$ for $w\in\Omega_+$ and
  $f(z)\in H^p(\mathbb{C}_+)$. In addition, $T_+^{-1}$ is bounded.
\end{proposition}

The proof of the inclusion part is nearly identical to the one in \cite{DL171}.
The boundedness of $T_+^{-1}$ comes from Proposition~\ref{pro-170827-2150}, 
Theorem~\ref{thm-170827-2120} and Banach open mapping theorem. The following
is the corresponding result for $T_-$.

\begin{proposition}\label{pro-170827-2223}
  If $1<p<\infty$, then for $T_-$ defined by \eqref{equ-170827-1820}, we have
  $H^p(\mathbb{C}_-)\subset T_-(H^p(\Omega_-))$, 
  or $T_-^{-1}(H^p(\mathbb{C}_-))\subset H^p(\Omega_-)$. Here
  $T_-^{-1}f(w)= f(\Psi_-(w))(\Psi_-'(w))^{\frac1p}$ for $w\in\Omega_-$ and
  $f(z)\in H^p(\mathbb{C}_-)$. In addition, $T_-^{-1}$ is bounded.
\end{proposition}

Before dealing with the $0<p\leqslant 1$ cases of the above two propositions, 
we need factorization theorems on $H^p(\mathbb{C}_\pm)$ which has been 
introduced in Lemma~\ref{lem-170801-1550}.

\begin{proposition}\label{pro-170828-0945}
  Propostion~\ref{pro-170827-2222} and Propostion~\ref{pro-170827-2223}
  are still true if $0<p\leqslant 1$. Besides, 
  \[\lVert T_+^{-1}\rVert\leqslant 5^{\frac1p}\quad
    \text{and}\quad
    \lVert T_-^{-1}\rVert\leqslant 6^{\frac1p},\]
  for all $0<p<\infty$.
\end{proposition}

The above propostion is proved in the same way as in \cite{DL172}. We also
need the factorization theorems on $H^p(\Omega_\pm)$ to extend 
Theorem~\ref{thm-170827-1120} and Theorem~\ref{thm-170827-1410} to the case
of $0< p\leqslant 1$. The following two corollaries of 
Lemma~\ref{lem-170801-1530} give the definitions of Blaschke product 
on $\Omega_\pm$.

\begin{corollary}\label{lem-170827-2210}
  Let $\{w_n\}$ be a sequence of points in $\Omega_+$, such that 
  \[\sum_{n=1}^\infty \frac{\mathrm{Im}\,\Psi_+(w_n)}{1+|\Psi_+(w_n)|^2}
      < \infty,\]
  and $m$ be the number of $\Psi_+(w_n)$ equal to $\mathrm{i}$. 
  Then the Blaschke product
  \[B_+(w)= \bigg(\frac{\Psi_+(w)-\mathrm{i}}{\Psi_+(w)+\mathrm{i}}\bigg)^m
      \prod_{\Psi_+(w_n)\neq\mathrm{i}} 
        \frac{|\Psi_+^2(w_n)+1|}{\Psi_+^2(w_n)+1}
      \cdot \frac{\Psi_+(w)-\Psi_+(w_n)}{\Psi_+(w)-\overline{\Psi_+(w_n)}},\]
  converges on $\Omega_+$, has non-tangential boundary limit $B_+(\zeta)$ 
  a.e.\@ on $\Gamma$, and the zeros of $B_+(w)$ are precisely 
  the points $w_n$, both counting multiplicity. Moreover, $|B_+(w)|<1$ 
  on $\Omega_+$ and $|B_+(\zeta)|=1$ a.e.\@ on $\Gamma$.
\end{corollary}

\begin{proof}
  This corollary of Lemma~\ref{lem-170801-1530} is obvious if we consider 
  the conformal mapping $w=\Phi_+(z)$ from $\mathbb{C}_+$ onto $\Omega_+$.
\end{proof}

\begin{corollary}\label{lem-170827-2220}
  Let $\{w_n\}$ be a sequence of points in $\Omega_-$, such that 
  \[\sum_{n=1}^\infty \frac{-\mathrm{Im}\,\Psi_-(w_n)}{1+|\Psi_-(w_n)|^2}
      < \infty,\]
  and $m$ be the number of $\Psi_-(w_n)$ equal to $-\mathrm{i}$. 
  Then the Blaschke product
  \[B_-(w)= \bigg(\frac{\Psi_-(w)+\mathrm{i}}{\Psi_-(w)-\mathrm{i}}\bigg)^m
      \prod_{\Psi_-(w_n)\neq -\mathrm{i}} 
        \frac{|\Psi_-^2(w_n)+1|}{\Psi_-^2(w_n)+1}
      \cdot \frac{\Psi_-(w)-\Psi_-(w_n)}{\Psi_-(w)-\overline{\Psi_-(w_n)}},\]
  converges on $\Omega_-$, has non-tangential boundary limit $B_-(\zeta)$ 
  a.e.\@ on $\Gamma$, and the zeros of $B_-(w)$ are precisely 
  the points $w_n$, both counting multiplicity. Moreover, $|B_-(w)|<1$ 
  on $\Omega_-$ and $|B_-(\zeta)|=1$ a.e.\@ on $\Gamma$.
\end{corollary}

Here comes the factorization theorem on $H^p(\Omega_+)$, see \cite{DL172} 
for proof, and that on $H^p(\Omega_-)$ is analogously stated and proved.

\begin{theorem}\label{thm-170827-2228}
  Let $0<p<\infty$, $F(w)\in H^p(\Omega_+)$, $F\not\equiv 0$, 
  $\{w_n\}$ be the zeros of $F(w)$, and $B_+(w)$ be the Blaschke product 
  associated with $\{w_n\}$. Then
  \[G(w)=\frac{F(w)}{B_+(w)}\in H^p(\Omega_+),\quad
    \text{and } \lVert F\rVert_{H^p(\Omega_+)}
    \leqslant \lVert G\rVert_{H^p(\Omega_+)}.\]  
\end{theorem}

The following theorem is one of our main results.
\begin{theorem}\label{thm-170827-2300}
  If $0<p<\infty$, $F(w)\in H^p(\Omega_+)$, then $F(w)$ has non-tangential 
  boundary limit $F(\zeta)\in L^p(\Gamma,|\mathrm{d}\zeta|)$ a.e.\@ 
  on $\Gamma$, $\lVert F\rVert_{L^p(\Gamma,|\mathrm{d}\zeta|)}
    \leqslant \lVert F\rVert_{H^p(\Omega_+)}$, 
  and $\lVert F_{\sigma-\tau,\tau}-\chi_{\Gamma}F
    \rVert_{L^p(\gamma,|\mathrm{d}\zeta|)}\to 0$ as $\tau\to 0$, 
  where $0<\tau<\sigma$. Besides, if $1\leqslant p< \infty$, then
  \[\frac1{2\pi\mathrm{i}} \int_{\Gamma} 
      \frac{F(\zeta)}{\zeta-w}\mathrm{d}\zeta
    = \left\{\!\!
        \begin{array}{ll}
          F(w) & \text{if $w\in \Omega_+$},\\
          0 & \text{if $w\in \Omega_-$},
        \end{array}\right.\]
\end{theorem}

\begin{proof}
  The $1<p<\infty$ case is Theorem~\ref{thm-170827-1120}. For general 
  $0<p<\infty$, the existence of non-tangential boundary limit and 
  $L^p(\gamma,|\mathrm{d}\zeta|)$ convergence are proved by the same method 
  as in~\cite{DL172}. We only need to prove the last equation under the 
  assumption that $p=1$. For $w_0\notin \Gamma$, there exists 
  $\delta\in(0,\sigma)$, such that $w_0\notin D_{\sigma+\delta,-\delta}
  \setminus\overline{D_{\sigma-\delta,\delta}}$. If $0<\tau<\frac\delta2$,
  then by Lemma~\ref{lem-170827-1000},
  \[\frac1{2\pi\mathrm{i}} \int_{\Gamma_{\sigma-\tau,\tau}} 
      \frac{F(w)}{w-w_0}\mathrm{d}w
    = \left\{\!\!
        \begin{array}{ll}
          F(w_0) & \text{if $w_0\in \Omega_+$},\\
          0 & \text{if $w_0\in \Omega_-$},
        \end{array}\right.\]
  
  Consider the same $I$ as in Theorem~\ref{thm-170827-1120}, that is 
  \begin{align*}
    I&= \int_{\Gamma_{\sigma-\tau,\tau}} 
          \frac{F(w)}{w-w_0}\mathrm{d}w
        - \int_{\Gamma} \frac{F(\zeta)}{\zeta-w_0}\mathrm{d}\zeta    \\
    &= \int_{\gamma_{\sigma+\tau,-\tau}} 
          F_{\sigma-\tau,\tau}(\zeta)\bigg(
            \frac1{\zeta_{\sigma-\tau,\tau}-w_0}
            - \frac1{\zeta-w_0}\bigg)\mathrm{d}\zeta                    \\
    &\qquad{}+ \int_{\gamma_{\sigma+\tau,-\tau}} 
          \frac1{\zeta-w_0}(F_{\sigma-\tau,\tau}(\zeta)
            - \chi_\Gamma F(\zeta))\mathrm{d}\zeta                    \\
    &= I_1+ I_2.
  \end{align*}
  From $|\zeta_{\sigma-\tau,\tau}-w_0|> \delta-\tau>\frac\delta2$ and
  $|\zeta-w_0|>\delta$, we have
  \[|I_1|\leqslant \frac{2\tau}{\delta^2} \int_\gamma 
       |F_{\sigma-\tau,\tau}(\zeta)||\mathrm{d}\zeta|
    \leqslant \frac{2\tau}{\delta^2} \lVert F\rVert_{H^1(\Omega_+)},\]
  and 
  \[|I_2|\leqslant \frac1\delta \lVert F_{\sigma-\tau,\tau}
       - \chi_\Gamma F\rVert_{L^1(\gamma,|\mathrm{d}\zeta|)},\]
  Then
  \[\lim_{\tau\to 0} |I|\leqslant \lim_{\tau\to 0} (|I_1|+|I_2|)= 0.\]
  and the last equation follows.
\end{proof}

Here follows the $H^p(\Omega_-)$ version of Theorem~\ref{thm-170827-2300}.

\begin{theorem}\label{thm-170827-2301}
  If $0<p<\infty$, $F(w)\in H^p(\Omega_-)$, then $F(w)$ has non-tangential 
  boundary limit $F(\zeta)\in L^p(\Gamma,|\mathrm{d}\zeta|)$ a.e.\@ 
  on $\Gamma$, $\lVert F\rVert_{L^p(\Gamma,|\mathrm{d}\zeta|)}
    \leqslant \lVert F\rVert_{H^p(\Omega_-)}$, and 
  $\lVert F_{\sigma+\tau,-\tau}
    - \chi_{\Gamma}F\rVert_{L^p(\gamma,|\mathrm{d}\zeta|)}\to 0$
  as $\tau\to 0$, where $0<\tau<\sigma$. Besides, if $1\leqslant p< \infty$, 
  then
  \[\frac1{2\pi\mathrm{i}} \int_{\Gamma} 
      \frac{F(\zeta)}{\zeta-w}\mathrm{d}\zeta
    = \left\{\!\!
        \begin{array}{ll}
          0 & \text{if $w\in \Omega_+$},\\
          -F(w) & \text{if $w\in \Omega_-$},
        \end{array}\right.\]
\end{theorem}

Proposition~\ref{pro-170827-1630} 
could now be extended to $1\leqslant p\leqslant\infty$, without changing 
the proof. Notice that functions in $H^{\infty}(\Omega_\pm)$ has 
non-tangential boundary limit a.e.\@ on $\Gamma$, since they could be 
transformed to functions in $H^\infty(\mathbb{C}_\pm)$.

\begin{corollary}
  For $1\leqslant p\leqslant\infty$, $\frac1p+\frac1q=1$, 
  if $F(w)\in H^p(\Omega_+)$, $G(w)\in H^q(\Omega_+)$, then
  \[\int_\Gamma F(\zeta)G(\zeta)\,\mathrm{d}\zeta= 0.\] 
\end{corollary}

\begin{corollary}\label{cor-170828-0900}
  If $0<p<q$, $F(w)\in H^p(\Omega_+)$ and 
  $F(\zeta)\in L^q(\Gamma,|\mathrm{d}\zeta|)$,
  then $F(w)\in H^q(\Omega_+)$.
\end{corollary}

\begin{proof}
  Choose $n\in\mathbb{N}$ such that $1<np<nq$, and write $F(w)= B(w)G(w)$ by
  Theorem~\ref{thm-170827-2228}, where $B(w)$ is the Blaschke product 
  associated with zeros of $F(w)$, and 
  $\lVert F\rVert_{H^p(\Omega_+)}\leqslant \lVert G\rVert_{H^p(\Omega_+)}$,
  then $|G(\zeta)|=|F(\zeta)|$ a.e.\@ on $\Gamma$, and 
  $G^{\frac1n}(w)\in H^{np}(\Omega_+)$. By Theorem~\ref{thm-170827-1120},
  \[G^{\frac1n}(w)= \frac1{2\pi\mathrm{i}} \int_\Gamma 
        \frac{G^{\frac1n}(\zeta)}{\zeta-w}\,\mathrm{d}\zeta\quad
    \text{for } w\in\Omega_+.\]
  Since $F(\zeta)\in L^q(\Gamma,|\mathrm{d}\zeta|)$, we have
  $G^{\frac1n}(\zeta)\in L^{nq}(\Gamma,|\mathrm{d}\zeta|)$ and,
  by Theorem~\ref{thm-170826-1440}, $G^{\frac1n}(w)\in 
    H^{nq}(\Omega_+)$. Then $G(w)\in H^q(\Omega_+)$, and it follows that
  $F(w)\in H^q(\Omega_+)$.
\end{proof}

The $H^p(\Omega_-)$ version of Corollary~\ref{cor-170828-0900} is stated as
follows.

\begin{corollary}
  If $0<p<q$, $F(w)\in H^p(\Omega_-)$ and 
  $F(\zeta)\in L^q(\Gamma,|\mathrm{d}\zeta|)$,
  then $F(w)\in H^q(\Omega_-)$.
\end{corollary}

Finally, we could prove that $H^p(\mathbb{C}_\pm)$ and $H^p(\Omega_\pm)$
are isomorphic if $0<p<\infty$. Remeber that $T_\pm$ below are defined 
in \eqref{equ-170827-1810} and \eqref{equ-170827-1820}.

\begin{theorem}\label{thm-170914-1100}
  If $0<p<\infty$, then $T_+\colon H^p(\Omega_+)\to H^p(\mathbb{C}_+)$ and
  $T_-\colon H^p(\Omega_-)\to H^p(\mathbb{C}_-)$ are both linear, 
  one-to-one, onto and bounded with
  \[5^{-\frac1p}\leqslant \lVert T_+\rVert\leqslant 1\quad
    \text{and}\quad
    6^{-\frac1p}\leqslant \lVert T_-\rVert\leqslant 1.\]
  Their inverses $T_\pm^{-1}$ are aslo bouned.
\end{theorem}

\begin{proof}
  In view of Proposition~\ref{pro-170827-2150}, 
  Proposition~\ref{pro-170828-0940} and Proposition~\ref{pro-170828-0945},
  we only need to prove that $T_\pm$ are bounded if $0<p\leqslant 1$, 
  which could be easily proved by Theorem~\ref{thm-170827-2300},
  Theorem~\ref{thm-170827-2301} and Fatou's lemma.
\end{proof}

\begin{corollary}
  If $0<p<\infty$, then $H^p(\Omega_\pm)$ are seperable.
\end{corollary}

\begin{proof}
  Suppose $F(w)\in H^p(\Omega_+)$, then $T_+F(z)\in H^p(\mathbb{C}_+)$ with
  $\lVert T_+F\rVert_{H^p(\mathbb{C}_+)}
    \leqslant \lVert F\rVert_{H^p(\Omega_+)}$. Since $H^p(\mathbb{C}_+)$ is
  seperable~\cite{De10}, if we let $\{f_n(z)\}$ be a countable dense subset
  of $H^p(\mathbb{C}_+)$, then for any $\varepsilon>0$, there exists 
  $f_N(z)$ such that $\lVert f_N- T_+F\rVert_{H^p(\mathbb{C}_+)}
    \leqslant \varepsilon$, which follows that
  \[\lVert T_+^{-1}f_N- F\rVert_{H^p(\Omega_+)}
    \leqslant \lVert T_+^{-1}\rVert 
        \lVert f_N- T_+F\rVert_{H^p(\mathbb{C}_+)}
    \leqslant \lVert T_+^{-1}\rVert \varepsilon.\]
  Thus $\{T_+^{-1}f_n(w)\}$ is a countable dense subset of $H^p(\Omega_+)$,
  and $H^p(\Omega_+)$ is seperable. The seperability of $H^p(\Omega_-)$ 
  could be proved by the same method.
\end{proof}

\section*{Funding}
This work is supported by National Natural Science Foundation 
of China(Grant No.\@ 11271045).

\end{document}